\documentclass[a4paper,12pt,reqno]{amsart}
\usepackage[T2A]{fontenc}
\usepackage[utf8x]{inputenc}
\usepackage[english]{babel}
\usepackage{amsmath}
\usepackage{amssymb}
\usepackage{amscd}
\usepackage{amsthm}

\makeatletter
\def\@settitle{\begin{center}%
    \baselineskip14\p@\relax
    \bfseries\textsc
    \@title
  \end{center}%
}
\makeatother

\usepackage{enumerate}

\usepackage{soul}
\usepackage{tikz,pgf}
\usetikzlibrary{calc}
\usetikzlibrary{arrows}

\newtheorem{propos}{Proposition}

\newtheorem{lemma}{Lemma}
\newtheorem{theorem}{Theorem}

\newtheorem{corollary}{Corollary}
\theoremstyle{definition}
\newtheorem{definition}{Definition}
\newtheorem{example}{Example}
\newtheorem{prooff}{Proof}
\theoremstyle{remark}
\newtheorem {remark}{Remark}

\DeclareMathOperator{\Spec}{Spec}

\DeclareMathOperator{\SL}{SL}

\DeclareMathOperator{\cone}{cone}

\DeclareMathOperator{\Supp}{Supp}

\def\GG{{\mathbb G}}

\def\CC{{\mathbb C}}
\def\KK{{\mathbb K}}
\def\TT{{\mathbb T}}
\def\ZZ{{\mathbb Z}}

\def\QQ{{\mathbb Q}}

\def\AA{{\mathbb A}}

\def\cone{\text{Cone}}
\def\Ga{\mathbb{G}_a}
\def\l{\lambda}
\def\m{\mu}
\newcommand{\mf}{\mathfrak}

\newcommand{\df}{\partial}
\newcommand{\ph}{\varphi}

\begin{document}
\title[Commuting homogeneous locally nilpotent derivations]{Commuting homogeneous locally nilpotent derivations}
\author{Dmitry Matveev}
\address{Lomonosov Moscow State University, Faculty of Mechanics and Mathematics, Department of Higher Algebra, Leninskie Gory 1, Moscow 119991, Russia \newline
\newline
\indent National Research University Higher School of Economics, Faculty of Computer Science, Kochnovskiy Proezd 3, Moscow, 125319 Russia
\newline}
\email{dmitry.a.matveev@yandex.ru}



\begin{abstract}
Let $X$ be an affine algebraic variety endowed with an action of complexity one of an algebraic torus $\TT$. It is well known that homogeneous locally nilpotent derivations on the algebra of regular functions $\KK[X]$ can be described in terms of proper polyhedral divisors corresponding to $\TT$-variety $X$. We prove that homogeneous locally nilpotent derivations commute if an only if some combinatorial criterion holds. These results are used to describe actions of unipotent groups of dimension two on affine $\TT$-varieties.
\end{abstract}

\maketitle

\section{Introduction}
\label{s1}

Let $\mathbb{K}$ be an algebraically closed field of characteristic zero, $\mathbb{T} = (\mathbb{K}^{\times})^n$ be an algebraic torus, $M = \mathfrak{X}(\mathbb{T}) \cong \mathbb{Z}^n$ be the character lattice of the torus, $N = \text{Hom}(M,\mathbb{Z})$ be the dual lattice to $M$, and $N_{\mathbb{Q}} = N \otimes \mathbb{Q}$.

A normal affine variety $X$ endowed with an effective regular action of the torus $\mathbb{T}$ is called $\mathbb{T}$-variety. The codimension of a general orbit of the action is said to be the complexity of the $\mathbb{T}$-action. Since the action of the torus is effective, with the notation above the complexity of the action on $X$ equals $ \dim X - \text{rank} \: M $.

If the complexity of the action is zero, then the torus acts on $ X $ with an open orbit and the variety is said to be toric. A combinatorial description of toric varieties is well known (see \cite{Fu} and \cite{CLS}). The description of affine $ \mathbb{T} $-varieties in the case of arbitrary complexity was obtained by K.~Altmann and J.~Hausen (see \cite{AH}). They proposed to assign $ \mathbb{T} $-varieties in terms of so-called proper polyhedral divisors.

Any affine toric variety corresponds to some polyhedral cone $ \sigma \subseteq N_{\mathbb{Q}} $. For an affine $ \mathbb{T} $-variety, an analogous description is more complicated. An affine $ \mathbb{T} $-variety corresponds to the triple $ (Y, \sigma, \mathfrak{D}) $, where $ Y $ is a normal semiprojective variety, $ \sigma $ is still a polyhedral cone, and $ \mathfrak{D} $ is a divisor on $ Y $ whose coefficients are polyhedra with a recession cone $\sigma$.

Let $ \mathbb{G}_a = (\mathbb{K}, +) $ be an additive group of the ground field, and let an affine $ \mathbb{T} $-variety $ X $ be endowed with an action of this group normalized by the torus $ \mathbb{T} $ in the automorphism group $ \text{Aut} (X) $. If general orbits of $ \mathbb{G}_a $-action  are contained in closures of general torus orbits, then it is said that the action is \textit{of vertical type}, and \textit{of horizontal type} otherwise.

It is known that $ \mathbb{G}_a $-actions on an affine variety $ X $ correspond to locally nilpotent derivations of the algebra $ \mathbb{K} [X] $ and $ \mathbb{G}_a $-actions normalized by a torus correspond to locally nilpotent derivations on $ \mathbb{K} [X] $ that are homogeneous with respect to the grading induced by the action of the torus $ \mathbb{T} $. The description of homogeneous locally nilpotent derivations in terms of proper polyhedral divisors and the so called Demazure roots is obtained by A.~Liendo for actions of vertical type of arbitrary complexity (see \cite{Li2}) and for actions of horizontal type of complexity one (see \cite{Li}, \cite{AL}). In addition, \cite{AL} describes pairs of $ \mathbb{G}_a $-subgroups corresponding to the root subgroups of the $ \text{SL}_2 \, (\KK) $ and $ \text{PSL}_2 \, (\KK) $ actions.

An objective of this paper is to describe $ \mathbb{G}_a^2 $-actions on the affine $ \mathbb{T} $-variety $ X $ in the same terms. In order to do that we need to describe pairs of commuting homogeneous locally nilpotent derivations.

Since locally nilpotent derivations can be either of vertical or of horizontal type, we have to consider three cases: \\* [\smallskipamount]
    (1) both derivations are of vertical type; \\
    (2) one derivation is of the vertical type, the another one is of horizontal type; \\
    (3) both derivations are of horizontal type.
\\* [\smallskipamount]
Case (1) is the simplest one and has already been investigated, for example, in \cite{Ro}. In cases (2) and (3) the torus action of complexity one is under consideration, because locally nilpotent derivations of horizontal type are classified only under this assumption. 

The existence of homogeneous locally nilpotent derivations of horizontal type imposes strong restrictions on $ Y $. The variety $ Y $ has to be isomorphic to $ \mathbb{A}^1 $ or $ \mathbb{P}^1 $. Thus, affine and projective subcases are distinguished.

The main result of this paper is a criteria to figure out whether homogeneous locally nilpotent derivations commute. The criteria is obtained in case (2) and in the affine subcase of case (3).

Sections 2-4 contain the necessary preliminary information. These sections are expounded in accordance with \cite{AL}. Section 2 is devoted to the combinatorial description of affine $ \mathbb{T} $-varieties in terms of proper polyhedral divisors. Section 3 explains the connection between locally nilpotent derivations on the algebra of regular functions of an affine $ \mathbb{T} $-variety and $ \mathbb{G}_a^2 $-actions on the variety. A description of locally nilpotent derivations in terms of proper polyhedral divisors is given in Section 4.

Sections 5-7 outline the main results of the paper, Section 5 deals with derivations of the vertical type, Section 6 investigates the conditions for commuting derivations of different types, Section 7 discusses derivations of horizontal type. Technical statements and auxiliary statements are included in the Appendix (Section 8).

The author is grateful to his supervisor Ivan~Arzhantsev for posing the problem and constant attention and support.

\section{Combinatorical desription of $\mathbb{T}$-varieties}\label{s2}

A complete combinatorial description of normal affine varieties with effective torus action is presented in \cite{AH}. In this section we introduce the necessary definitions and formulate the main results of this paper.

Let $ M = \ZZ^n $ be a lattice of rank $ n $ and $ N = \text{Hom} (M, \mathbb{Z}) $ be its dual lattice. Consider the spaces $ M_{\mathbb{Q}} = M \otimes \mathbb{Q} $ and $ N_{\mathbb{Q}} = N \otimes \mathbb{Q} $. The natural pairing is defined as follows $ M_{\mathbb{Q}} \times N_{\mathbb{Q}} \to \mathbb{Q} $, $ (m, p) \mapsto \langle m, p \rangle = p (m) $.

Let $ \sigma $ be a polyhedral cone in $ N_{\mathbb{Q}} $. The set of polyhedra in $ N_{\mathbb{Q}} $ that can be represented as a Minkowski sum of a cone $ \sigma $ and some polytope $ \Pi $ will be denoted as $ \text{Pol}_{\sigma} (N_{ \mathbb{Q}}) $.

A cone
$$
\sigma^{\lor} = \{m \in M ​​_{\mathbb{Q}} \mid \langle m, u \rangle \ge 0, \; \forall u \in \sigma \} \subset M_{\mathbb{Q}}
$$
is said to be the dual cone of $ \sigma \subseteq N_{\mathbb{Q}} $.

To each polyhedron $ \Delta \in \text{Pol}_{\sigma} (N_{\mathbb{Q}}) $ we associate the support function $ h_{\Delta}: \sigma^{\lor} \to \mathbb{ Q} $ defined as follows:
$$
h_{\Delta} (m) = \min \langle m, \Delta \rangle: = \min_{\substack{p \in \Delta}} \langle m, p \rangle.
$$
The minimum is attained because $ m $ is nonnegative on $ \sigma $. Since the function is considered on a polyhedron, the minimum is always reached at the vertex. Denoting $ \{v_i \} $ by the set of vertices of the polyhedron $ \Delta $, we obtain the expression
$$
h_{\Delta} (m) = \min_{\substack{i}} \{v_i (m) \} \quad \text{for all } m \in \sigma^{\lor}.
$$


The definition of a projective variety over a scheme is given in \cite [Chapter 4]{Ha}. The normal total variety $ Y $ constructed over the scheme is called \textit{semiprojective}.

An irreducible subvariety of codimension one of the variety $X$ is called a prime divisor. The elements of the free Abelian group generated by all prime divisors of the variety $ X $ are called the Weil divisors. We say that the divisor is principal if it is a divisor of some rational function on the variety $ X $. A Weil divisor is called a Cartier divisor or a locally principal divisor if there exists a covering of the variety $ X $ by open sets $ U_i $ such that the restriction of the divisor to every $ U_i $ is a principal divisor. A formal sum of prime divisors with rational coefficients is called $ \QQ $-Cartier divisor if some its multiple is a Cartier divisor.

\begin{definition} \label{pp-div}
We call \textit{$ \sigma $-polyhedral divisor} on $ Y $ a formal sum $ \mathfrak{D} = \sum_{Z \subseteq Y} \Delta_Z \cdot Z $, where $ Z $ are prime divisors on $ Y $, polyhedra $ \Delta_Z \in \text{Pol}_{\sigma} (N_{\mathbb{Q}}) $ and $ \Delta_Z = \sigma $ for all but finitely many~$ Z $. For each $ m \in \sigma^{\lor} $, we can evaluate the divisor $ \mathfrak{D} $ at the point $ m $:
$$
\mathfrak{D} (m) = \sum_{\substack{Z \subseteq Y}} h_Z (m) \cdot Z,
$$
where $ h_Z $ is the support function of the polyhedron $ \Delta_Z $. A $ \sigma $-polyhedral divisor is said to be \textit{proper} if the following two conditions are satisfied: \\* [\smallskipamount]
(1) $ \mathfrak{D} (m) $ is semiample and $ \mathbb{Q} $-Cartier for all $ m \in \sigma^{\lor}; $ \\
(2) $ \mathfrak{D} (m) $ is big for any $ m \in \text{rel.int} (\sigma^{\lor}) $. \\* [\smallskipamount]
Here $ \text{rel.int} (\sigma^{\lor}) $ denotes the relative interior of the cone $ \sigma^{\lor} $. A $ \mathbb{Q} $- Cartier divisor $ D \subseteq Y $ is said to be \textit{semiample} if there exists an integer $ r> 0 $ such that $ rD $ is base point free (that is the divisor has no point lying in the intersection of the supports of all the divisors that are linearly equivalent to $D$ ), and is called \textit{big} if for some $ r> 0 $ there is a linearly equivalent to $ rD $ divisor $ D_0 $ such that $ Y \setminus \Supp D_0 $ is affine.
\end{definition}

Let $ \mathbb{T} = \Spec \mathbb{K} [M]$ be an $ n $-dimensional algebraic torus with the character lattice $ M $ and $ X = \Spec A $ be an affine $ \mathbb{T} $-variety. The torus action $ \mathbb{T} \times X \to X $ corresponds to a homomorphism of algebras $ A \to A \otimes \mathbb{K} [M] $ that defines $ M $-grading on $ A $. Conversely, every $ M $-grading corresponds to the action of the torus $ \mathbb{T} = \Spec \mathbb{K} [M] $. We denote by $ \sigma^{\lor}_M $ the semigroup $ \sigma^{\lor} \cap M $ and by $ \chi^m $ the character of the torus $ \mathbb{T} $ corresponding to the vector $ m $ of the lattice $ M $.

For a rational divisor $ D $ we consider the set
$$
H^0 (U, \mathcal{O}_Y (D)) = \{f \in \KK (Y) \mid \text{div} \, f |_U + D |_U \ge 0 \}.
$$

Now we are ready to formulate the theorem (\cite [Section 3]{AH}) describing $ \mathbb{T} $-varieties in terms of proper polyhedral divisors.

\begin{theorem} \label{AH-Th}
To any proper $ \sigma $-polyhedral divisor $ \mathfrak{D} $ on a semiprojective variety $ Y $ one can associate a normal affine $ \mathbb{T} $-variety $ X [Y, \mathfrak{D}] = \textup{Spec } \: A [Y, \mathfrak{D}] $ of dimension $ \textup{rank} \: M + \textup{dim} \: Y $, where
$$
A [Y, \mathfrak{D}] = \bigoplus_{\substack{m \in \sigma^{\lor}_M}} A_m \chi^m \quad \text{and} \quad A_m = H^0 ( Y, \mathcal{O}_Y (\mathfrak{D} (m)) \subseteq \mathbb{K} (Y).
$$
Conversely, any normal affine $ \mathbb{T} $-variety is isomorphic to $ \mathbb{T} $-variety $ X [Y, \mathfrak{D}] $ for some semiprojective variety $ Y $ and some proper $ \sigma $-polyhedral divisor $ \mathfrak{D} $ on $ Y $.
\end{theorem}

The description is not unique, but uniqueness can be achieved by imposing some minimality  conditions on the pair $ (Y, \mathfrak{D}) $ (see \cite [Section 8]{AH}).

\begin{corollary} \label{AH-crl}
Let $ \mathfrak{D} $ and $ \mathfrak{D} '$ be the proper $ \sigma $-polyhedral divisors on the normal semiprojective variety $ Y $. If for any prime divisor $ Z $ in $ Y $ there exists a vector $ v_Z \in N $ such that
$$
\mathfrak{D} = \mathfrak{D} '+ \sum_{\substack{Z}} (v_Z + \sigma) \cdot Z \quad \text{and} \quad \text{divisor} \sum_{\substack{ Z}} \langle m, v_Z \rangle \cdot Z \text{main} \forall m \in \sigma^{\lor}_M,
$$
then the variety $ X [Y, \mathfrak{D}] $ is equivariantly isomorphic to $ X [Y, \mathfrak{D} '] $.
\end{corollary}

\begin{example} \label{ex01}
We consider the one-dimensional lattice $ \ZZ $ of the characters of the torus $ \TT = \KK^{\times} $, the cone $ \sigma = \{0 \} $ and the divisor $ \mf{D} = [0,1] \cdot \{0 \} $ on the affine line $ Y = \AA^1 $. Then
$$
\mf{D} (m) = \begin{cases}
0 \cdot \{0 \}, & m \ge 0 \\
m \cdot \{0 \}, & m <0;
\end{cases}
$$
$$
A_m = \begin{cases}
\mathbb{K} [t], & m \ge 0 \\
t^{-m} \mathbb{K} [t], & m <0.
\end{cases}
$$
We note that $ x = \chi $ and $ y = t \chi^{-l} $ are generators of the algebra $ A = \bigoplus_{\substack{m}} A_m \chi^m $ that do not satisfy any algebraic relation. Thus, these combinatorial data define an affine plane with the action of the torus $ \TT $ on $ \KK [x, y] $, where $ \lambda \circ x = \lambda x $ and $ \lambda \circ y = \lambda^{-1} y $.
\end{example}

\begin{example}\label{ex02}

Consider the example proposed by A.~Liendo in \cite{Li}. Suppose $ N = \ZZ^2 $ and $ \sigma = \{(0,0) \} $. In $ N_{\mathbb{Q}} = \mathbb{Q}^2 $ we take the triangle $ \Delta_0 $ with the vertices $ (0,0) $, $ (0,1) $, $ (- 1/4 , -1) $ and the divisor $ \Delta_1 = \{0 \} \times [0,1] $ (Figure 1).

\begin{center}
\begin{tikzpicture}
\draw[gray,-latex] (-1.5,0) -- (1.5,0);
\draw[gray,-latex] (0,-1.5) -- (0,1.5);
\fill (0,0) -- (0,1.2) -- (-0.3,-1.2) -- (0,0);
\draw[densely dotted] (-0.3,-1.2) -- (-0.3,0);
\draw[densely dotted] (-0.3,-1.2) -- (0,-1.2);
\draw (0.2,1.2) node {\small $1$};
\draw (0.25,-1.2) node {\small $-1$};
\draw (-0.5,0.3) node {\small $-\frac{1}{4}$};
\draw (1,0.4) node {\small $\Delta_0 \subseteq N_{\mathbb{Q}}$};
\end{tikzpicture} \qquad \qquad \qquad \qquad
\begin{tikzpicture}
\draw[gray,-latex] (-1.5,0) -- (1.5,0);
\draw[gray,-latex] (0,-1.5) -- (0,1.5);
\draw[very thick] (0,0) -- (0,1.25);
\draw (0.25,1.25) node {\small $1$};
\draw (1,0.4) node {\small $\Delta_1 \subseteq N_{\mathbb{Q}}$};
\end{tikzpicture}
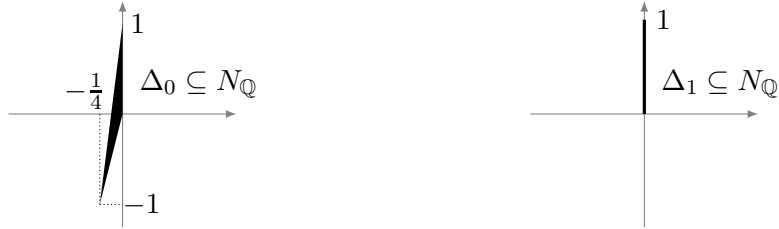
\\ \small{\textsc{Fig. 1}}
\end{center}
\vspace{0.5em}

Consider $ Y = \mathbb{A}^1 $ and $ \mf{D} = \Delta_0 \cdot \{0 \} + \Delta_1 \cdot \{1 \} $. Figure 2 shows $ M_{\mathbb{Q}} $ divided into 4 sectors depending on the support functions $ h_{\Delta_0} $ and $ h_{\Delta_1} $ values on $ m = (m_1, m_2) \in \sigma^{\lor} = M_{\mathbb{Q}} $.

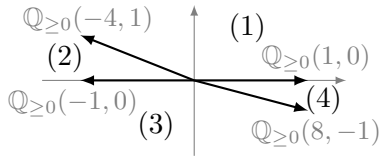
\begin{center}
\begin{tikzpicture}
\draw[gray,-latex] (-2,0) -- (2,0);
\draw[gray,-latex] (0,-1) -- (0,1);
\draw[-latex, thick] (0,0) -- (-1.5,0.6);
\draw[-latex, thick] (0,0) -- (-1.5,0);
\draw[-latex, thick] (0,0) -- (1.5,0);
\draw[-latex, thick] (0,0) -- (1.5,-0.4);
\draw (0.7,0.7) node {\small $(1)$};
\draw (-0.5,-0.6) node {\small $(3)$};
\draw (-1.7,0.3) node {\small $(2)$};
\draw (1.7,-0.25) node {\small $(4)$};
\draw[gray] (1.6,-0.7) node {\footnotesize $\mathbb{Q}_{\ge 0}(8,-1)$};
\draw[gray] (1.6,0.3) node {\footnotesize $\mathbb{Q}_{\ge 0}(1,0)$};
\draw[gray] (-1.6,-0.3) node {\footnotesize $\mathbb{Q}_{\ge 0}(-1,0)$};
\draw[gray] (-1.4,0.8) node {\footnotesize $\mathbb{Q}_{\ge 0}(-4,1)$};
\end{tikzpicture} \qquad
\begin{tikzpicture}
\draw (0,1.7) node[anchor=west] {(1) $h_{\Delta_0}(m) = -\frac{1}{4} m_1 - m_2$, $h_{\Delta_1}(m) = 0$};
\draw (0,1.1) node[anchor=west] {(2) $h_{\Delta_0}(m) = 0$, $h_{\Delta_1}(m) = 0$};
\draw (0,0.5) node[anchor=west] {(3) $h_{\Delta_0}(m) = m_2$, $h_{\Delta_1}(m) = m_2$};
\draw (0,-0.1) node[anchor=west] {(4) $h_{\Delta_0}(m) = -\frac{1}{4} m_1 - m_2$, $h_{\Delta_1}(m) = m_2.$};
\end{tikzpicture}
\\ \small{\textsc{Fig. 2}}
\end{center}
\vspace{0.5em}

By calculating $ A_m = H^0 (Y, \mathcal{O}_Y (\mathfrak{D} (m)) $ for each case, we obtain the following collection of generators of the algebra $ A [\mathbb{A}^1, \mathfrak{D}] = \bigoplus_{\substack{m}} A_m \chi^m $:
$$
u_1 = -t \chi^{(4,0)}, \quad
u_2 = \chi^{(- 1,0)}, \quad
u_3 = - \chi^{(- 4,1)}, \quad
u_4 = t (t-1) \chi^{(8, -1)}. \quad
$$
It is easy to verify that the generators satisfy the unique irreducible relation \mbox{$ u_1 + u_1^2u_2^4 + u_3u_4 = 0 $}. Thus, there is an isomorphism
\begin{equation} \label{ex2-alg-iso}
A [\mathbb{A}^1, \mathfrak{D}] \cong \mathbb{K} [x_1, x_2, x_3, x_4] / (x_1 + x_1^2x_2^4 + x_3x_4)
\end{equation}
with the following $ \ZZ^2 $-grading on the algebra $ A [\AA^1, \mathfrak{D}] $:
$$
\deg x_1 = (4,0), \: \deg x_2 = (-1,0), \: \deg x_3 = (-4,1), \: \deg x_4 = (8, -1).
$$

\end{example}

\section{Locally nilpotent derivations and $\mathbb{G}_a$-actions}
\label{s3}

Let $X = \text{Spec}\,A$ be an affine variety over the ground field $\KK$. By $\mathbb{G}_a$ we denote an additive group of the field $\mathbb{K}$. A derivation $\df$ on $A$ is said to be locally nilpotent if for any $a \in A$ there exists $n \in \ZZ_{> 0}$ such that $\df^n(a)=0$. A locally nilpotent derivation $\df$ on $A$ corresponds to an action of $\mathbb{G}_a$ on $A$ defined by $\phi_{\df}: \mathbb{G}_a\times A \to A$, $\phi_{\df}: (t,f) \mapsto \exp(t\df)(f)$. Then it follows that locally nilpotent derivation corresponds to an action of $\mathbb{G}_a$ on $X=\text{Spec}\,A$. And vice versa an action of $\mathbb{G}_a$ on $X$ arises from some locally nilpotent derivation on $A$ (see \cite[Section 1.5]{Fr}).

Let $A = A [Y, \ mf{D}]$ be the algebra with $M$-grading from the theorem~\ref{AH-Th}. A locally nilpotent derivation $\df$ is said to be homogeneous with respect to $M$-grading if it sends homogeneous elements into homogeneous elements. In this case, the degree $\deg \df = \deg \df(f) - \deg f$ of a locally nilpotent derivation $\df$ is well defined for any homogeneous element $f \in A \setminus \ker \df$. It is said that the $\mathbb{G}_a$-action is normalized by the torus $\mathbb{T}$ if the group $\mathbb{G}_a$ is normalized by the torus as a subgroup in the automorphism group $\text{Aut}(X)$. The $\mathbb{G}_a$-action is normalized by the torus $\mathbb{T}$ if and only if the corresponding locally nilpotent derivation $\df$ is homogeneous with respect to $M$-grading.

Next, consider the group $\mathbb{G}_a^2$ acting on the variety $X$. Since $\mathbb{G}_a^2 = \mathbb{G}_a \times \mathbb{G}_a$ specifying the actions of this group is equivalent to specifying the actions of two commuting groups $\mathbb{G}_a$. In addition, if $\mathbb{G}_a^2$ is normalized by the torus $\mathbb{T}$ then there is an action of $\mathbb{T}$ on the Lie algebra $\text{Lie}(\mathbb{G}_a^2) = \mathfrak{g}$. The action of the torus $\mathbb{T}$ is diagonalizable on $\mathfrak{g}$. Therefore, in $\mathfrak{g}$ we can choose an eigen basis $\{x_1, x_2\}$. Then $\mathfrak{g}$ splits into the direct sum of one-dimensional components: $ \mathfrak{g} = \mathfrak{g}_1 \oplus \mathfrak{g}_2 = \langle x_1 \rangle \oplus \langle x_2 \rangle$, and $\{\text{exp}\,(tx_i)\mid t \in \mathbb{K}\}$ are isomorphic to $\mathbb{G}_a$ and are normalized by the torus. Thus, the assignment of a $\mathbb{T}$-normalizable $\mathbb{G}_a^2$-action to $X$ is equivalent to specifying a pair of commuting homogeneous locally nilpotent derivations.

Any homogeneous locally nilpotent derivation $\df$ on the algebra $A$ can be extended to the quotient field $\text{Quot}(A)$. The action of the torus $\mathbb{T}$ on $\text{Quot}(A)$ extends similarly. Consider the subfield of the rational invariants $\text{Quot}(A)^{\mathbb{T}} = \mathbb{K}(Y)$. It is said that locally nilpotent derivation $\df$ is of \textit{vertical type} if $\df(\mathbb{K}(Y)) = 0$, and of \textit{horizontal type} otherwise. The fact that derivation $\df$ is of vertical type geometrically means that the typical orbits of the corresponding $\mathbb{G}_a$-action lie in the $\mathbb{T}$-action typical orbits closures.

\section{Locally nilpotent derivations on affine $\mathbb{T}$-varieties}
\label{s4}

Firstly let us consider the classification of homogeneous locally nilpotent derivations of vertical type described on $\mathbb{T}$-varieties of arbitrary complexity described in \cite{Li2}. By $\sigma(1)$ denote the set of rays (a facet of dimension 1) of cone $\sigma$. Further we suppose $\rho \in \sigma(1)$ denoting a ray and its primitive vector of the lattice.

\begin{definition}\label{dem}
Let $\sigma$ be a pointed cone in $N_{\mathbb{Q}}$. A vector $e \in M$ is called \textit{Demazure root} of the cone $\sigma$, if the following conditions hold: \\*[\smallskipamount]
(1) there exist $\rho_e \in \sigma(1)$ such that $\langle e, \rho_e \rangle = -1;$ \\
(2) $\langle e, \rho \rangle \ge 0$ for all $\rho \in \sigma(1)\setminus \{\rho_e\}$.\\*[\smallskipamount]
A set of all Demazure roots of $\sigma$ denotes by $\mathcal{R}(\sigma)$. And we say that the ray $\rho_e$ is associated with $e$. To simplify the notation we will sometimes write $\rho$ for $\rho_e$.
\end{definition}

Let $A$ be an algebra $A[Y,\mathfrak{D}]$ and $\mathfrak{D}=\sum_{Z} \Delta_Z \cdot Z$ be $\sigma$-polyhedral divisor on semiprojective variety $Y$. By $\{v_{i, Z}\}$ denote a set of vertices of polyhedron $\Delta_Z$ for prime divisor $Z \subseteq Y$. Let $e$ be a Demazure root of $\sigma$. Then we consider
$$
\mathfrak{D}(e) = \sum_{\substack{Z}} \min_{\substack{i}}\{v_{i, Z}(e)\}\cdot Z \quad \text{and} \quad \Phi_e^{\times} = H^0(Y,\mathcal{O}_Y(\mathfrak{D}(e)))\setminus\{0\}.
$$
For every function $\varphi \in \Phi_e^{\times}$ we set
\begin{equation}\label{d_vert_formula}
\df_{e,\ph}(f\chi^m) = \langle m, \rho_e \rangle \cdot \ph \cdot f\chi^{m+e} \quad \text{for all } m \in \sigma_M^{\lor} \text{ and } f \in \mathbb{K}(Y).
\end{equation}
This expression gives a derivation on $A$.
The next theorem provides a classification for homogeneous locally nilpotent derivations of vertical type on the variety $A[Y,\mathfrak{D}]$ (see \cite[Theorem~2.4]{Li2} and also \cite[Theorem~1.7]{AL}).

\begin{theorem}\label{ver}
For each $e \in \mathcal{R}(\sigma)$ and function $\varphi \in \Phi_e^{\times}$ the derivation $\df_{e,\ph}$ is a homogeneous locally nilpotent derivation of vertical type  $A = A[Y,\mathfrak{D}]$ with degree $e$ and kernel
$$
\ker \df_{e,\ph} = \bigoplus_{\substack{m \in \tau_e \cap M}} A_m\chi^m,
$$
where $\tau_e \subseteq \sigma^{\lor}$ is a facet of $\sigma^{\lor}$  dual to $\rho_e$. Conversely, if $\df \ne 0$ is a homogeneous locally nilpotent derivation of fiber type on $A$, then $\df$ is equal to $\df_{e,\ph}$ for some root $e \in \mathcal{R}(\sigma)$ and some function $\varphi \in \Phi_e^{\times}$.
\end{theorem}

The classification of locally nilpotent derivations of horizontal type is more complex and is known only in complexity one case. In this case $Y$ from the triple $(Y, \sigma, \mathfrak{D})$ is a smooth curve. As described in \cite{Li} homogeneous locally nilpotent derivations exist if $Y$ is isomorphic either to $\mathbb{A}^1$ or to $\mathbb{P}^1$. Within the frame of current research we consider only the case of affine line.

\begin{definition}\label{def-hor-1}
A \textit{colored} $\sigma$-polyhedral divisor $Y$ is a collection $\widehat{\mathfrak{D}} = \{ \mathfrak{D}; v_z \: \forall z \in Y \}$, where:
\begin{enumerate}
\item $\mathfrak{D} = \sum_{z \in Y} \Delta_z \cdot z$ is a proper $\sigma$-polyhedral divisor on $Y$, and $v_z$ is a vertex of $\Delta_z;$
\item $v_{\deg} := \sum_{z}v_z$ is a vertex of $\deg \mathfrak{D} := \sum_{z}\Delta_z;$ 
\item $v_z \in N$ with at most one exception in \textit{marked point}~$z_0$. If all the vertices $v_z$ belong to the lattice $N$ ny point of the divisor $\mathfrak{D}$ can be chosen for the marked point $z_0$.
\end{enumerate}
\end{definition}

Let $\widehat{\mathfrak{D}}$ be a colored $\sigma$-polyhedral divisor on $Y$ and $\omega \subseteq N_{\mathbb{Q}}$ be a cone generated by $\deg \mathfrak{D} - v_{\deg}$. By $\widehat{\omega} \subseteq (N\oplus\mathbb{Z})_{\mathbb{Q}}$ we denote \textit{an associated cone} of the divisor $\mathfrak{D}$ generated by $(\omega,0)$ and $(v_{z_0},1)$. Let also $d$ be a positive integer such that $d \cdot v_{z_0} \in N$. 

\begin{definition}\label{def-hor-2}
A pair $(\widehat{\mathfrak{D}}, e)$ where $\widehat{\mathfrak{D}}$ is a colored $\sigma$-polyhedral divisor on $Y$ and $e \in M$ is called \textit{coherent} if the following conditions hold:
\begin{enumerate}
    \item there exists $s \in \mathbb{Z}$ such that $\hat{e} = (e,s) \in M\oplus\mathbb{Z}$ is a Demazure root of assosiated cone $\widehat{\omega}$ with associated ray $\hat{\rho} = (d\cdot v_{z_0},d)$. In this case $s = -1/d-v_{z_0}(e);$
    \item $v(e)\ge 1+v_z(e)$ for each $z \ne z_0$ and for each $v \ne v_z$ of polyhedron $\Delta_z;$ 
    \item $d\cdot v(e)\ge 1+d\cdot v_{z_0}(e)$ for each $v \ne v_{z_0}$ of polyhedron $\Delta_{z_0}$.
\end{enumerate}
\end{definition}

Consider $L = \{m\in M \mid v_{z_0}(m) \in \ZZ \}$. Since minimum function is linear with respect to the Minkowski sum, $\mathfrak{D}(m)$ is linear for all $m \in \omega^{\lor}$. Since every divisor is principal on a line $Y = \AA^1$ , there exist linear functions $\ph^m \in \mathbb{K}(Y)$ such that $\text{div}(\varphi^m) + \mathfrak{D}(m) = 0$ and $\varphi^m \cdot \varphi^{m'} = \varphi^{m+m'}$ for all $m, m' \in \omega^{\lor}_L$.

The following theorem (see \cite[Theorem 3.28]{Li}, and also \cite[Theorem 1.10]{AL}) gives a classification for homogeneous locally nilpotent derivations of horizontal type on $A[Y,\mf{D}]$.
\begin{theorem}\label{hor}
Let $X = X[Y,\mf{D}]$ be a normal affine $\mathbb{T}$-variety and $Y = \AA^1$. Then homogeneous locally nilpotent derivations of horizontal type on the algebra $A[Y,\mf{D}]$ are in one-to-one correspondence with coherent pairs $(\widehat{\mathfrak{D}}, e)$, where $\widehat{\mathfrak{D}}$ is a colored $\sigma$-polyhedral divisor and $e \in M$. Moreover, homogeneous locally nilpotent derivation $\df$ corresponding to $(\widehat{\mathfrak{D}}, e)$ has degree $e$ and kernel
$$
\ker \df = \bigoplus_{\substack{m \in \omega^{\lor}_L}} \mathbb{K}\varphi^m.
$$
\end{theorem}
Let us recall a formula for homogeneous locally nilpotent derivation corresponding to $(\widehat{\mathfrak{D}}, e)$. Without loss of generality we suppose $z_0 = 0$. Since every divisor is principal on a line we can assume $v_z = 0 \in N$ for all $z \ne 0$ by corollary~\ref{AH-crl}. If we consider $\mathbb{K}[Y]$ as $\mathbb{K}[t]$ we obtain the following formula for $\df$ (see \cite[Section~1]{AL}):
\begin{equation}\label{f-hor}
\df_e(t^r\chi^m) = d(v_{0}(m)+r)\cdot t^{r+s} \cdot \chi^{m+e} \quad \text{for all } (m,r) \in M\oplus\mathbb{Z}.  
\end{equation}
Notice that if $\hat{m} = (m,r) \in M\oplus\ZZ$ and $\chi^{\hat{m}} = \chi^m\cdot t^r$ then $\df$  is similar to the derivation of vertical type:
$$
\df_e(\chi^{\hat{m}}) = \langle\hat{m}, \hat{\rho} \rangle \cdot \chi^{\hat{m}+\hat{e}} \quad \text{for all } \hat{m} \in M\oplus\mathbb{Z}.
$$

\begin{example}\label{ex01d}
Consider the derivations for the variety of example \ref{ex01} defined by the divisor $\mf{D} = [0,1]\cdot\{0\}$ on affine line.
Firstly, since $\sigma = \{0\}$ it follows that there is no Demazure roots. Consequently there is no derivations of vertical type.
We have two cases for the derivations of horizontal type:

\textit{Case 1.} $v_0 = \{0\}$. \\ Here $\omega = \QQ_{\ge 0}$, $\widehat{\omega} = \QQ^2_{\ge 0}$, $d=1$, $\hat e = (e,s)$, $s=-1$, $e \ge 0$. From Condition 3 of Corollary~\ref{def-hor-2} it follows that $e\ge 1$. Consequently,  $\df(t^r\chi^m)  = r \cdot t^{r-1} \chi^{m+1}$. We have $x = \chi$, $y = t\chi^{-1}$ and $\df_e(x) = 0$,  $\df_e(y) = x^{e-1}$.

\textit{Case 2.} $v_0 = \{1\}$. \\ Here $\omega = \QQ_{\le 0}$, $\widehat{\omega} = \text{Cone}\bigl((-1,0),(1,1)\bigr)$, $d=1$, $\hat e = (e,s)$, $s=-1-e$, $e \ge 0$. Condition 3:  $0 \ge 1+e$,  that is $s\ge0$ and $e\le-1$. Consequently,   $\df(t^r\chi^m)  = (m+r) \cdot t^{r+s} \chi^{m+e}$. Then we obtain, $\df_e(x) = \df_e(\chi)=t^s\chi^{e+1} = t^s\chi^{-s} = y^s$,  $\df_e(y) = \df_e(t\chi^{-1}) = 0$.

Thus, all homogeneous locally nilpotent derivations on the algebra with $\ZZ$-grading $\KK[x,y]$, $\deg x = 1$, $\deg y = -1$ are given by 
\begin{equation}\label{ex1}
x^a\frac{\df}{\df y}\quad \text{and} \quad y^b\frac{\df}{\df x},\quad \text{where } \: a,b \in \ZZ_{\ge 0}.
\end{equation}
\end{example}

\begin{example}\label{ex02d}
Consider the derivations for the variety of example \ref{ex02}. There exist 6 possibilities to choose a pair of vertices for colored divisor $\widehat{\mathfrak{D}}$. From Condition 2 of Corollary~\ref{def-hor-2} it follows that only 4 cases are possible:
\begin{align*}
    (1)\:\: v_0 &= (0,0), \tilde v_0 = (0,0) \quad     (3)\:\: v_2 = (-1/4,-1), \tilde v_0 = (0,0)\\
    (2)\:\: v_1 &= (0,1), \tilde v_1 = (0,1) \quad     (4)\:\: v_2 = (-1/4,-1), \tilde v_1 = (0,1).
\end{align*}
The point $z_0 = 0$ is a marked point in all cases. We have
\begin{align*}
    &(1) \:\: d=1,\: \widehat \omega = \cone\{(0,1,0),(-1/4,-1,0),(0,0,1)\}        \\
    &(2) \:\: d=1,\: \widehat \omega = \cone\{(0,-1,0),(-1/8,-1,0),(0,1,1)\}       \\
    &(3) \:\: d=4,\: \widehat \omega = \cone\{(0,1,0),(1/4,1,0),(-1/4,-1,1)\}       \\
    &(4) \:\: d=4,\: \widehat \omega = \cone\{(0,1,0),(1/8,1,0),(-1/4,-1,1)\}.
\end{align*}
By inequalities for $\hat e = (e,s) = (a,b,s)$, $a,b,s \in \ZZ$ in Demazure root definition and conditions 2 and 3 of Corollary~\ref{def-hor-2} we obtain
\begin{align*}
    &(1) \:\: \df_1(t^r\chi^m) = r\cdot t^{r-1}\chi^{m+e}, \:s=-1,\: b \ge 1,\: a+4b \le -4       \\
    &(2) \:\: \df_2(t^r\chi^m) = (r+m_2)\cdot t^{r+s}\chi^{m+e}, \:b+s=-1,\: b \le -1,\: a+8b\le -4      \\
    &(3) \:\: \df_3(t^r\chi^m) = (4r-m_1-4m_2)\cdot t^{r+s}\chi^{m+e}, \\
        &\qquad a+4b-4s=1,\: b \ge 1,\:a+8b\ge1  \\
    &(4) \:\: \df_4(t^r\chi^m) = (4r-m_1-4m_2)\cdot t^{r+s}\chi^{m+e}, \\
        &\qquad a+4b-4s=1,\: b \le -1,\: a+4b \ge 1.
\end{align*}
Derivations (1) and (3) can be easily expressed in variables $x_1, x_2, x_3, x_4$:
\begin{align*}
    (1) \:\: &\df_1 = x_2^{-a-4b-4}x_3^b \frac{\df}{\df x_1}-x_2^{-a-4b-4}x_3^{b-1}(2 x_1 x_2^4 + 1)\frac{\df}{\df x_4} ,\\
    &b \ge 1,\: a+4b \le -4    \\
    (3) \:\: &\df_3 = x_1^{\frac{1}{4}(a+4b-1)}x_3^b\frac{\df}{\df x_2}-4x_1^{\frac{1}{4}(a+4b+7)}x_2^3x_3^{b-1}\frac{\df}{\df x_4} ,\\
    &b \ge 1,\:a+4b \ge 1, \: a = 4k+1, \, k \in \ZZ.    
\end{align*}
For derivations (2) and (4) the similar formulae are quite complicated. Therefore, as A.~Liendo (see \cite[Example 3.35]{Li}) we need to choose another system of generators for the algebra. Let us consider the divisor $\mf{D} = \Delta_0\cdot\{0\} + \Delta_1'\cdot\{1\}$, where $\Delta_1 = \{0\} \times [-1,0]$. By Corollary~\ref{crl_main} this divisor defines an isomorphic algebra of functions and $\TT$-variety. In this case we have
$$
x_1 = -t\chi^{(4,0)}, \quad
x_2 = \chi^{(-1,0)}, \quad
x_3 = (1-t)\chi^{(-4,1)}, \quad
x_4 = t\chi^{(8,-1)}. \quad
$$
Notice that under this substitution the associated cones and the basic formulae for derivations $\df_2$ and $\df_4$ have not changed.
In variables $x_1, x_2, x_3, x_4$ we obtain the following:
\begin{align*}
    (2) \:\: &\df_2 = -x_2^{-a-8b-4}x_4^{-b} \frac{\df}{\df x_1}+(x_2^{-a-8b-4}x_4^{-b-1}+2x_1x_2^{-a-8b}x_4^{-b-1})\frac{\df}{\df x_3},\\ 
    &b \le -1,\: a+8b\le -4 \\
    (4) \:\: &\df_4 = x_1^{\frac{1}{4}(a+8b-1)}x_4^{-b}\frac{\df}{\df x_2} - 4x_1^{\frac{1}{4}(a+8b+7)} x_2^3 x_4^{-b-1}\frac{\df}{\df x_3}, 
    \\ &b \le -1,\: a+8b \ge 1, \: a = 4k+1, \, k \in \ZZ.
\end{align*}
\end{example}

\section{Commuting derivations of vertical type}
\label{s5}

In this section we recall a criteria  figuring out whether two homogeneous locally nilpotent derivations of vertical type commute or not.

\begin{lemma}\label{lem1}
The commutator of homogeneous locally nilpotent derivations of vertical type  $\df_{e,\ph}$ and $\df_{\tilde e,\tilde \ph}$ is given by
\begin{eqnarray}
\nonumber [\df_{e,\ph},\df_{\tilde e,\tilde \ph}] = \bigl( \langle m, \tilde \rho \rangle \langle \tilde e, \rho \rangle - \langle m, \rho \rangle \langle e, \tilde \rho \rangle \bigr) \cdot \ph \tilde \ph \cdot f\chi^{m+e+\tilde e},
\end{eqnarray}
where $\rho$ and $\tilde \rho$ are associated with roots $e$ and~$\tilde e$.
\end{lemma}
\begin{proof}
Using $\ph$, $\tilde \ph$ $\in \Phi^{\times}_e \subseteq \mathbb{K}(Y) \subseteq \ker \df_{e,\ph} \cap \ker \df_{\tilde e,\tilde \ph}$, we get $\df_{e,\ph}(\tilde \ph) = \df_{\tilde e,\tilde \ph}(\ph) = 0$. Taking into account (\ref{d_vert_formula}), we obtain
\begin{eqnarray}
\nonumber \df_{e,\ph} \df_{\tilde e,\tilde \ph}(f\chi^m) = \df_{e,\ph} (\langle m, \rho_{\tilde e} \rangle \cdot \tilde \ph \cdot f\chi^{m+\tilde e}) = \langle m, \rho_{\tilde e} \rangle \cdot \langle m, \rho_e \rangle \cdot  \ph \tilde \ph \cdot f\chi^{m+e+\tilde e}.
\end{eqnarray}
Similarly, we have
$$
\df_{\tilde e,\tilde \ph}\df_{e,\ph}(f\chi^m) = \langle m, \rho_e \rangle \cdot \langle m, \rho_{\tilde e} \rangle \cdot  \ph \tilde \ph \cdot f\chi^{m+e+\tilde e}.
$$
This completes the proof.
\end{proof}

\begin{propos}\label{prop1}
Let $\df_{e,\ph}$ and $\df_{\tilde e,\tilde \ph}$ be homogeneous locally nilpotent derivations of vertical type on the algebra $A = A[Y,\mf{D}]$, where $Y$ is a semiprojective variety and $\mf{D}$ is a proper polyhedral divisor. Then $\df_{e,\ph}$ and $\df_{\tilde e,\tilde \ph}$ commute if and only if one of the following conditions holds:
\\*[\smallskipamount]
    \textup{(1)} \; $ \langle \tilde e, \rho \rangle = \langle e, \tilde \rho \rangle = 0 $, where $\rho$ and $\tilde \rho$ are rays associated with roots $e$ and~$\tilde e$;\\
    \textup{(2)} \; $\rho = \tilde \rho$.
\end{propos}

\begin{proof}
By Lemma~\ref{lem1} the derivations commute if an only if
\begin{equation*}\label{f-vv}
\langle m, \tilde \rho \rangle \langle \tilde e, \rho \rangle - \langle m, \rho \rangle \langle e, \tilde \rho \rangle = 0 \quad \text{for all } m \in M.
\end{equation*}
Since dual pairing is linear in each of its arguments, we have
$$
\bigl\langle m, \langle \tilde e, \rho \rangle \cdot \tilde \rho   - \langle e, \tilde \rho \rangle \cdot \rho \bigr\rangle= 0 \quad \text{for all } m \in M.
$$
It follows that $\langle \tilde e, \rho \rangle \cdot \tilde \rho  = \langle e, \tilde \rho \rangle \cdot \rho $. By evaluating pairing on $e$ and $\tilde e$, we obtain
$$
\begin{cases}
\langle \tilde e, \rho \rangle \cdot \bigl( \langle e, \tilde \rho \rangle + 1 \bigr) = 0\\
\langle e, \tilde \rho \rangle \cdot \bigl( \langle \tilde e, \rho \rangle + 1 \bigr) =0.
\end{cases}
$$
Hence, either $\langle \tilde e, \rho \rangle = 0$ or $\langle e, \tilde \rho \rangle = -1$ holds. These conditions are equivalent to those ones in the statement of the lemma.
\end{proof}

\begin{remark}
For $\rho$ and $\tilde \rho$ there exist roots $e$ and $\tilde e$ such that $ \langle e, \rho \rangle = \langle \tilde  e, \tilde \rho \rangle = -1 $ and $ \langle \tilde e, \rho \rangle = \langle e, \tilde \rho \rangle = 0 $ if and only if the primitive vectors $\rho$ and $\tilde \rho$ can be extended to a basis of the lattice $N$.
\end{remark}

\begin{remark}
The general orbits of $\mathbb{G}_a^2$-actions corresponding to pairs of commuting locally nilpotent derivations of case (2) preposition~\ref{prop1} are 1-dimensional. Indeed, the kernels of these two locally nilpotent derivations are equal. And by \cite[Proposition 3.4]{VP} regular invariants separate general orbits of a unipotent group action on an affine variety in this case. At the same time, commuting derivations in case (1) have different kernels, therefore general $\mathbb{G}_a^2$-orbits are 2-dimensional.
\end{remark}

\section{Commuting derivations of vertical and horizontal type}
\label{s6}

Here we suppose $Y$ to be an affine line $\mathbb{A}^1$. By Theorem~\ref{hor} a locally nilpotent derivation of horizontal type is defined by $(\widehat{\mathfrak{D}}, \tilde e)$, where $\widehat{\mathfrak{D}}$ is a colored proper polyhedral divisor, and for $\tilde e \in M$ there exists an integer $\tilde s$ such that $\hat e = (\tilde e,\tilde s) \in M\oplus\mathbb{Z}$ is a root of the associated cone $\widehat \omega$. In appropriate coordinates we have
$$
\df_{\tilde e}(t^r\chi^m) = d(v_0(m)+r)\cdot t^{r+\tilde s} \cdot \chi^{m+\tilde e} \quad \text{for all } (m,r) \in M\oplus\mathbb{Z}.
$$

\begin{lemma}\label{lem2}
Let $\df_{e,\ph}$ and $\df_{\tilde e}$ be homogeneous locally nilpotent derivations of vertical and horizontal type respectively on the algebra $A = A[Y,\mf{D}]$. Then the following conditions hold:
\\*[\smallskipamount]
    \textup{(1)} \; $[\df_{e,\ph},\df_{\tilde e}](t) = d \cdot \langle \tilde e, \rho_{e} \rangle \cdot \ph \cdot t^{s+1} \chi^{e + \tilde e}$ \\
    \textup{(2)} \; $[\df_{e,\ph},\df_{\tilde e}](\chi^m) = d\cdot \Bigl( v_0(m) \cdot \langle \tilde e, \rho_{e} \rangle \cdot \ph  -  \langle m, \rho_{e} \rangle \cdot \bigl(\ph' \cdot t + \ph \cdot  v_0(e) \bigr) \Bigr) \cdot t^s\chi^{m+e+\tilde e}$.
\end{lemma}

\begin{proof}
The first condition follows from the equality
$$
[\df_{e,\ph},\df_{\tilde e}](t) = \df_{e,\ph}(\df_{\tilde e}(t))-\df_{\tilde e}(\df_{e,\ph}(t)) =  \df_{e,\ph}(d\cdot t^{s+1}\chi^{\tilde e}) - 0 = d \cdot \langle \tilde e, \rho_{e} \rangle \cdot \ph \cdot t^{s+1} \chi^{e + \tilde e}.
$$
To prove the second one let us compute $\df_{e,\ph}(\df_{\tilde e}(\chi^m))$ and $\df_{\tilde e}(\df_{e,\ph}(\chi^m))$. Next we have
\begin{multline*}
    \df_{e,\ph}(\df_{\tilde e}(\chi^m)) = \df_{e,\ph}(d\cdot v_0(m) \cdot t^s \chi^{m+\tilde e}) = d\cdot v_0(m) \cdot \langle m+\tilde e, \rho_{e} \rangle \cdot \ph \cdot t^s\chi^{m+e+\tilde e} = \\
\shoveright{
    = d\cdot v_0(m) \cdot \bigl( \langle m, \rho_{e} \rangle + \langle \tilde e, \rho_{e} \rangle \bigr) \cdot \ph \cdot t^s\chi^{m+e+\tilde e} }\\
\shoveleft
    \df_{\tilde e}(\df_{e,\ph}(\chi^m)) = \df_{\tilde e}(\langle m, \rho_{e} \rangle \cdot \ph \cdot \chi^{m+e}) =  \\
    = \langle m, \rho_{e} \rangle \cdot \bigl(\df_{\tilde e}(\ph) \cdot \chi^{m+e} + \ph \cdot d \cdot v_0(m+e) \cdot t^s\chi^{m+e+\tilde e} \bigr).
 \end{multline*}
By the chain rule for the derivative of composition of two functions it follows that 
$$
\df_{\tilde e}(\ph) = \ph' \cdot \df_{\tilde e}(t) =  \ph' \cdot d \cdot t^{s+1}\chi^{\tilde e},
$$ 
where prime denotes the partial derivative with respect to $t$. Therefore, we obtain
$$
\df_{\tilde e}(\df_{e,\ph}(\chi^m)) = d  \cdot \langle m, \rho_{e} \rangle \cdot  t^s\chi^{m+e+\tilde e} \cdot \bigl(\ph' \cdot t + \ph \cdot  v_0(m+e) \bigr).
$$
Subtracting $\df_{\tilde e}(\df_{e,\ph}(\chi^m))$ from $\df_{e,\ph}(\df_{\tilde e}(\chi^m))$ we obtain the second condition.

\end{proof}

\begin{theorem}\label{t4}
Let $\df_{e,\ph}$ and $\df_{\tilde e,\tilde \ph}$ be homogeneous locally nilpotent derivations of vertical and horizontal type respectively on the algebra $A = A[Y,\mf{D}]$. Then $\df_{e,\ph}$ and $\df_{\tilde e,\tilde \ph}$ commute if and only if one of the following conditions holds:
 \\*[\smallskipamount]
\textup{(1)} \; $\langle \tilde e, \rho_{e} \rangle = 0$; \\
\textup{(2)} \; $\ph = c t^{-v_0(e)} \in \Phi_e^{\times}, \; c \in \mathbb{K}, v_0(e) \in \ZZ.$
\end{theorem}

\begin{proof}
If the commutator is equal to zero, the first expression of Lemma~\ref{lem2} is equal to zero as well:
$$
[\df_{e,\ph},\df_{\tilde e}](t) = d \cdot \langle \tilde e, \rho_{e} \rangle \cdot \ph \cdot t^{s+1} \chi^{e + \tilde e} = 0.
$$
If $\ph \equiv 0$ then the derivation of vertical type is identically zero and the statement is clear. Otherwise $\ph$, $t$, $\chi$ are not identically zero. Further, $d$ is a positive integer. Hence, the commutator applied to function $t$ vanishes if and only if the pairing $\langle \tilde e, \rho_{e} \rangle$ equals zero.

Taking into account this remark we obtain the following expression for the second condition of Lemma~\ref{lem2}:
$$
[\df_{e,\ph},\df_{\tilde e}](\chi^m) = -d  \cdot \langle m, \rho_{e} \rangle \cdot  t^s\chi^{m+e+\tilde e} \cdot \bigl(\ph' \cdot t + \ph \cdot  v_0(e) \bigr).
$$
The last expression is equal to zero if and only if
$$
\ph' \cdot t + \ph \cdot  v_0(e) = 0.
$$
But this is a separable differential equation in $\ph$. Recall that the solution is a family of exponential functions $c t^{-v_0(e)}$, where $c \in \mathbb{K}$. Finally, $\ph$ belongs to $\Phi_e^{\times}$ by definition. This concludes the proof. 
\end{proof}

\begin{remark}
The condition $c t^{-v_0(e)} \in \Phi_e^{\times}$  means that ${\mathfrak{D}}(e) - v_0(e) \cdot \{0\} \ge 0$.
\end{remark}

\begin{remark}
Under the restrictions of Theorem~\ref{t4} the kernels of homogeneous locally nilpotent derivations do not coincide and general $\mathbb{G}_a^2$-orbits on $X$ are 2-dimensional.
\end{remark}

\section{Commuting derivations of horizontal type}
\label{s7}

As before, we suppose $Y = \mathbb{A}^1$. One difficulty in working with derivations of horizontal type is that different derivations can be simplified as in~(\ref{f-hor}) in different systems of generators. Let $\df_{e}$ and $\df_{\tilde e}$ be homogeneous locally nilpotent derivations of horizontal type corresponding to  $(\widehat{\mathfrak{D}}_v, e)$ and $(\widehat{\mathfrak{D}}_{\tilde v}, \tilde e)$ respectively, where 
$$
\widehat{\mathfrak{D}}_v = \{ \mathfrak{D}=\sum_{z \in  \mathbb{A}^1} \Delta_z \cdot z; v_z \} \text{  and  } \widehat{\mathfrak{D}}_{\tilde v} = \{ \mathfrak{D}=\sum_{z \in  \mathbb{A}^1} \Delta_z \cdot z; \tilde{v}_z \},
$$
and $z_0$, $\tilde{z}_0$ are marked points. Further we suppose $\tilde{z}_0$ = 0. By Corollary~\ref{AH-crl} we can assume that $v_z$ = 0 for all $z \ne z_0$. Hence, by introducing $q=t-z_0$ we have
$$
\df_e(q^r\chi^m) = d(v_{z_0}(m)+r)\cdot q^{r+s} \cdot \chi^{m+e} \quad \text{for all } (m,r) \in M\oplus\mathbb{Z}.
$$
To obtain the similar expression for $\df_{\tilde e}$, let us consider $\mathfrak{D}' = \mathfrak{D} - \sum_{z\ne 0} (\tilde{v}_z+\sigma)\cdot z$. By Corollary~\ref{AH-crl} there exists an isomorphism of varieties $X[\mathbb{A}^1,\mathfrak{D}] \to X[\mathbb{A}^1,\mathfrak{D}']$ given by the family of functions $\ph^m \in \mathbb{K}(t)$ with divisors $\mathfrak{D}'(m)-\mathfrak{D}(m)$ and $\ph^m \cdot \ph^{m'} = \ph^{m+m'}$. As a result $\tilde{v}'_z = 0$ where $z\ne 0$ for the divisor $\mathfrak{D}'$. Thus, we obtain
$$
\df_{\tilde e}(t^r\cdot\ph^m\chi^m) = \tilde{d}(\tilde{v}_{0}(m)+r)\cdot t^{r+\tilde{s}} \cdot \ph^{m+\tilde{e}}\chi^{m+\tilde{e}} \quad \text{for all } (m,r) \in M\oplus\mathbb{Z}.
$$
On Figure~3 colored vertices for marked points of corresponding divisor are marked with asterisks:
\begin{center}
\begin{tikzpicture}
\draw[gray,-latex] (-5,0) -- (5,0);
\draw (5,0) node[anchor=north] {$\AA^1$};
\path (0,0) coordinate (origin);
\path (3,0) coordinate (right);
\path (-3,0) coordinate (left);

\filldraw  (0,0) circle (1pt) 
                    node[anchor=north] {\small $z_1 = \tilde z_0 = 0$}
           (3,0) circle (1pt) 
                    node[anchor=north] {\small $z_2$}
           (-3,0) circle (1pt) 
                    node[anchor=north] {\small $z_0$};

\draw (-3,0.6) -- 
      ++(-0.5,0) --
      ++(0,0.5);
\draw[dotted]   (-3,0.6) --
                ++(0.5,0) --
                ++(0,0.5);
\draw (-3,1.6) -- 
      ++(0.5,0) --
      ++(0,-0.5);
\draw[dotted]   (-3,1.6) --
                ++(-0.5,0) --
                ++(0,-0.5);
\draw (0,0.6) -- 
      ++(-0.5,0) --
      ++(0,0.5);
\draw[dotted]   (0,0.6) --
                ++(0.5,0) --
                ++(0,0.5);
\draw (0,1.6) -- 
      ++(0.5,0) --
      ++(0,-0.5);
\draw[dotted]   (0,1.6) --
                ++(-0.5,0) --
                ++(0,-0.5);
\draw (3,0.6) -- 
      ++(-0.5,0) --
      ++(0,0.5);
\draw[dotted]   (3,0.6) --
                ++(0.5,0) --
                ++(0,0.5);
\draw (3,1.6) -- 
      ++(0.5,0) --
      ++(0,-0.5);
\draw[dotted]   (3,1.6) --
                ++(-0.5,0) --
                ++(0,-0.5);

\filldraw  (-3.5,0.6) circle (1pt) 
                    node[anchor=east] {\small $v_{z_0}$}
           (-2.5,1.6) circle (1pt) 
                    node[anchor=west] {\small $\tilde v_{z_0}$}
           (-0.5,0.6) circle (1pt) 
                    node[anchor=east] {\small $v_{z_1}$}
           (0.5,1.6) circle (1pt) 
                    node[anchor=west] {\small $\tilde v_0$}
           (2.5,0.6) circle (1pt) 
                    node[anchor=east] {\small $v_{z_2}$}
           (3.5,1.6) circle (1pt) 
                    node[anchor=west] {\small $\tilde v_{z_2}$};
                    
\draw (-3.5,0.6) node[anchor=north] {\small $\star$}
      (0.5,1.6) node[anchor=south] {\small $\star$};
                    
\draw[gray] (-3,0) node[anchor=south] {\footnotesize $\phantom{_{z_0}}\Delta_{z_0}$}
            (0,0) node[anchor=south]  {\footnotesize $\phantom{_{z_0}}\Delta_{z_1}$}
            (3,0) node[anchor=south]  {\footnotesize $\phantom{_{z_0}}\Delta_{z_2}$};
                    
\end{tikzpicture}
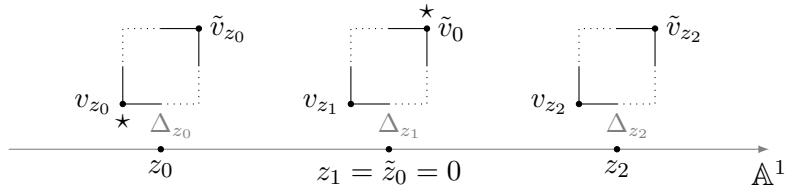
\\ \small{\textsc{Fig. 3}}
\end{center}

\begin{definition}
    A system of generators of the algebra $A$ is said to be \textit{associated} with derivations $\df_{e}$ and $\df_{\tilde e}$ if these derivations is given by
    \begin{align*}%
    \df_e(q^r\chi^m) &= d(v_{z_0}(m)+r)\cdot q^{r+s} \cdot \chi^{m+e}, \quad q = t - z_0 \\
    \df_{\tilde e}(t^r\cdot\ph^m\chi^m) &= \tilde{d}(\tilde{v}_{0}(m)+r)\cdot t^{r+\tilde{s}} \cdot \ph^{m+\tilde{e}}\chi^{m+\tilde{e}}.
    \end{align*}
\end{definition}

\begin{definition}
    A derivation $\df$ is said to be \textit{simple} in the fixed system of generators of the algebra $A$, if all the vertices $v$ of the pair $(\widehat{\mathfrak{D}}_v, e)$ is equal to zero with at most one exception.
\end{definition}

Note that $\df_{e}$ is always simple in the system of generators \textit{associated} with derivations $\df_{e}$ and $\df_{\tilde e}$.

The proof of the next lemma is given in Appendix, because it is rather long and technical. By $\alpha_m$ we denote $t(\ph^m)'/\ph^m$.

\begin{lemma}\label{lm3}
Let $\df_{e}$ and $\df_{\tilde e}$ be homogeneous locally nilpotent derivations of horizontal type on the algebra $A = A[Y,\mf{D}]$ with vectors of parameters $(e,v_{z_0},v_z,d)$, $z \ne z_0$, and $(\tilde e,\tilde v_{0},\tilde v_z,\tilde d)$, $z \ne 0$ respectively. Then the following equalities hold in the system of generators associated with $\df_{e}$ and $\df_{\tilde e}$:
\begin{align}
&[\df_{e},\df_{\tilde e}](t) = d \tilde d \cdot t^{\tilde s} q^s \cdot \ph^{\tilde e}\chi^{e+\tilde e} \cdot B,\label{com-t}\\
&[\df_{e},\df_{\tilde e}](\chi^m) = d\tilde d\cdot t^{\tilde s-1} q^{s-1} \cdot \ph^{\tilde e}\cdot \chi^{m+e+\tilde e} \bigl(C_0 + C_{1} + C_{2}\bigr), \label{com-chi}
\end{align} where
\begin{align}
\label{dt}
B_{\phantom{0}} &= (\tilde{v}_0(e)-\tilde s - v_{z_0}(\tilde e)+s)t - (\tilde{v}_0(e)-\tilde s-1)z_0 - (\alpha_e+\alpha_{\tilde e})q  \\*[\smallskipamount]
 C_{0} &= \tilde s\tilde{v}_0(m)q^2  - sv_{z_0}(m)t^2 +\bigl(\tilde{v}_0(m) v_{z_0}(\tilde e)- v_{z_0}(m)\tilde{v}_0(e)\bigr)t q   \label{C_0}\\
 C_{1} &= \tilde{v}_0(m)\alpha_{\tilde e}q^2-\tilde s\alpha_m q^2 - v_{z_0}(\tilde e)\alpha_m tq + v_{z_0}(m)\alpha_e t q    \label{C_1}\\
 C_{2} &= -t\alpha'_m q^2 - \alpha_m\alpha_{\tilde e}q^2.   \label{C_2}
\end{align}
\end{lemma}

Further we need an explicit form for $\alpha_m$.
By Corollary~\ref{AH-crl} we have 
$$
\text{div}(\ph^m) = \mathfrak{D}'(m)-\mathfrak{D}(m) = - \sum_{z\ne 0} (\tilde{v}_z+\sigma)\cdot z.
$$
It follows that $\ph^m = \prod_{z \ne 0}(t-z)^{-\tilde{v}_z(m)}$ and
\begin{equation}\label{am}
\alpha_m = -\sum_{z \ne 0} \tilde{v}_z(m) \dfrac{t}{t-z}, \quad 
\alpha_m' = \sum_{z \ne 0} \tilde{v}_z(m) \dfrac{z}{(t-z)^2}.
\end{equation}
Let us enumerate all points $z$ such that $\Delta_z \ne \sigma$ with indexes from $1$ to $l$, $z_k \ne 0, \,z_k \ne z_0$. We introduce the following notation: $\mu(t) = \prod_{k=1}^l (t-z_k)$, $\mu_k(t) = \prod_{i\ne k} (t-z_i)$. Then we obtain
\begin{equation}\label{amm}
\alpha_m = -t\cdot \sum_{z_k \ne 0} \tilde{v}_{z_k}(m)\dfrac{\mu_k(t)}{\mu(t)}, \quad 
\alpha_m' = \sum_{z_k \ne 0} z_k\tilde{v}_{z_k}(m) \dfrac{\mu_k(t)^2}{\mu(t)^2}.
\end{equation}
Besides, from Definition~\ref{def-hor-2} (2) it follows that 
\begin{equation}\label{s}
s = -1/d - v_{z_0}(e), \quad
\tilde s = -1/\tilde{d} - \tilde{v}_0(\tilde e).
\end{equation}
Since $v_z = 0$ for all $z \ne z_0$, expressions $\tilde v_z(e) \ge 1 + v_z(e)$ and $v_z(\tilde e) \ge 1 + \tilde v_z(\tilde e)$ of Definition~\ref{def-hor-2} (2) can be simplified:
\begin{equation}\label{ineq-v_z}
\tilde v_z(e) \ge 1, \quad
\tilde v_z(\tilde e) \le -1
\end{equation}
for all $\tilde v_z \ne 0$. Our aim is to find necessary and sufficient conditions of vanishing the commutators from Lemma~\ref{lm3}. We first make several simplifying assumptions. 

\begin{definition}
    Homogeneous locally nilpotent derivations $\df_{e}$ and $\df_{\tilde e}$ are called \textit{adjacent}, if colored vertices and marked points of corresponding pairs $(\widehat{\mathfrak{D}}_v, e)$ and $(\widehat{\mathfrak{D}}_{\tilde v},\tilde e)$ are the same.
\end{definition}

Note that $z_0 = \tilde z_0 = 0$ and $v_{z_0} = \tilde v_{\tilde z_0} = v_0$ in an associated system of generators in this case.

\begin{propos}\label{th_3_1}
    Let $\df_{e}$ and $\df_{\tilde e}$ be homogeneous locally nilpotent derivations of horizontal type on the algebra $A = A[Y,\mf{D}]$
    
    Then $\df_{e}$ and $\df_{\tilde e}$ commute if and only if the following conditions hold in the system of generators associated with $\df_{e}$ and $\df_{\tilde e}$:
    \begin{enumerate}
        \item $\tilde v_z = 0$ for all $z \ne 0;$
        \item $v_0 \in N$, $\tilde v_0 \in N$, there exists a point of the divisor $z_1$ such that $\tilde v_{z_1}(e) = 1$, $\tilde v_{z_1}(\tilde e) = -1$ and all $\tilde v_z = 0$ for $z \ne 0$, $z \ne z_1$.
    \end{enumerate}
\end{propos}

\begin{proof}
    We can simplify expressions for commutators obtained in Lemma~\ref{lm3} under the additional conditions of the proposition. Firstly, note that by definition $d = \tilde d$ and then by (\ref{s}) we obtain $\tilde s + v_{0}(\tilde e) = s + v_{0}(e) = -1/d$. Using Lemma~\ref{lm3} we have $B =  -(\alpha_e+\alpha_{\tilde e})t$. The application of Lemma~\ref{lm4} yields that $B=0$ if and only if $\tilde v_z(e)+\tilde v_z(\tilde e) = 0$ for all $z \ne 0$. Then (\ref{C_0}--\ref{C_2}) is given by
    \begin{align*}
        C_{0} &= 0 \\
        C_{1} &= 1/d \cdot \alpha_m t^2 \\
        C_{2} &= -t\alpha'_m t^2 - \alpha_m\alpha_{\tilde e}t^2.
    \end{align*}

    The commutator vanishes if and only if $C_0 + C_1 + C_2 = 0$. Using (\ref{amm}) for all expressions with $\alpha$ we obtain
    \begin{align*}
           1/d \cdot t \cdot\sum_{k = 1}^l  \tilde{v}_{z_k}(m)\dfrac{\mu_k(t)}{\mu(t)} - t \cdot\sum_{k = 1}^l  z_k\tilde{v}_{z_k}(m)\dfrac{\mu_k(t)^2}{\mu(t)^2} -\\- t^2 \cdot\sum_{k = 1}^l  \tilde{v}_{z_k}(m)\dfrac{\mu_k(t)}{\mu(t)}\cdot\sum_{k = 1}^l  \tilde{v}_{z_k}(\tilde e)\dfrac{\mu_k(t)}{\mu(t)} = 0.
    \end{align*}
    Then for the numerator of this fraction we obtain
    \begin{align*}
           1/d \cdot t \cdot\sum_{k = 1}^l  \tilde{v}_{z_k}(m)\mu_k(t)\mu(t) - t \cdot\sum_{k = 1}^l  z_k\tilde{v}_{z_k}(m)\mu_k(t)^2 -\\- t^2 \cdot\sum_{k = 1}^l  \tilde{v}_{z_k}(m)\mu_k(t)\cdot\sum_{k = 1}^l  \tilde{v}_{z_k}(\tilde e)\mu_k(t) = 0.
    \end{align*}
    Substituting $z_j$ for $t$ we have $\tilde{v}_{z_j}(m)(1+\tilde{v}_{z_j}(\tilde e))=0$. The coefficient at the leading monomial should be zero.  It follows that $\sum_{k = 1}^l\tilde{v}_{z_j}(m)\bigl(1+d\sum_{k = 1}^l\tilde{v}_{z_j}(\tilde e)\bigr)=0$. Using (\ref{ineq-v_z}) we get $\tilde{v}_{z_j}(\tilde e) \le -1$ for $\tilde{v}_{z_j} \ne 0$. 
    
    Here only two cases are possible. The first one is that $\tilde v_{z_j}=0$. This corresponds to the condition~1 of the proposition. The other possibility is that there exists a point $z_1 \ne 0$ such that $\tilde{v}_{z_1}(\tilde e) = -1$, $d=1$, $\tilde v_{z_j} = 0$ for $j \ne 1$. Besides, using $\tilde v_z(e)+\tilde v_z(\tilde e) = 0$ we have $\tilde{v}_{z_1}(e) = 1$. By this we obtain the second condition of the proposition. 
   The sufficiency of conditions is verified by a direct substitution.
\end{proof}

Let us consider coherency conditions (2) and (3) of Definition~\ref{def-hor-2} for $(\widehat{\mathfrak{D}}, e)$. Let $\df_{e}$ and $\df_{\tilde e}$ be derivations with vertices $V = \{v_{z_0}, \dots, v_{z_l}\}$ and $\tilde V = \{\tilde v_{z_0}, \dots, \tilde v_{z_l} \}$, where $z_0$ and $\tilde z_0$ are corresponding marked vertices ($\tilde z_0 = z_k$ for some $k$). Then conditions (2) and (3) is given by
\begin{align*}
    \tilde v_{z_k}(e) &\ge 1 + v_{z_k}(e), \quad \text{for } \tilde v_{z_k} \ne v_{z_k},\; k \ne 0 \\
    v_{z_k}(\tilde e) &\ge 1 + \tilde v_{z_k}(\tilde e), \quad \text{for } v_{z_k} \ne \tilde v_{z_k},\; k \ne 0 \\
    d\tilde v_{z_0}(e) &\ge 1 + d v_{z_0}(e), \quad \text{for } \tilde v_{z_0} \ne v_{z_0} \\
    \tilde d v_{\tilde z_0}(\tilde e) &\ge 1 + \tilde d \tilde v_{\tilde z_0}(\tilde e), \quad \text{for } v_{\tilde z_0} \ne \tilde v_{\tilde z_0}.
\end{align*}
The first and the second correspond to not marked points, the third and the second one correspond to marked ones. Thus, each point $z_k$ is associated with two \textit{coherency inequalities}: the first one depends on $e$,  the second one depends on $\tilde e$.

\begin{definition}
     Homogeneous locally nilpotent derivations $\df_{e}$ and $\df_{\tilde e}$ corresponding to pairs $(\widehat{\mathfrak{D}}_v, e)$ and $(\widehat{\mathfrak{D}}_{\tilde v}, \tilde e)$ with vertices $V = \{v_{z_0}, \dots, v_{z_l}\}$ and $\tilde V = \{\tilde v_{z_0}, \dots, \tilde v_{z_l} \}$, are called \textit{coherent}, if for every point~$z_k$ one of the following conditions holds: $v_{z_k} = \tilde v_{z_k}$ or both of coherency inequalities for $z_k$ become equalities.
\end{definition}

Figure~4 illustrates this definition. If a coherency inequality becomes an equality then we draw an arrow starting from a corresponding vertex. If a point $z_k$ is marked then an arrow is dotted.

\begin{center}
\begin{tikzpicture}
\draw[gray,-latex] (-5,0) -- (5,0);
\draw (5,0) node[anchor=north] {$\AA^1$};
\path (0,0) coordinate (origin);
\path (3,0) coordinate (right);
\path (-3,0) coordinate (left);

\filldraw  (0,0) circle (1pt) 
                    node[anchor=north] {\small $z_1 = \tilde z_0$}
           (3,0) circle (1pt) 
                    node[anchor=north] {\small $z_2$}
           (-3,0) circle (1pt) 
                    node[anchor=north] {\small $z_0$};

\draw (-3,0.6) -- 
      ++(-0.5,0) --
      ++(0,0.5);
\draw[dotted]   (-3,0.6) --
                ++(0.5,0) --
                ++(0,0.5);
\draw (-3,1.6) -- 
      ++(0.5,0) --
      ++(0,-0.5);
\draw[dotted]   (-3,1.6) --
                ++(-0.5,0) --
                ++(0,-0.5);
\draw (0,0.6) -- 
      ++(-0.5,0) --
      ++(0,0.5);
\draw[dotted]   (0,0.6) --
                ++(0.5,0) --
                ++(0,0.5);
\draw (0,1.6) -- 
      ++(0.5,0) --
      ++(0,-0.5);
\draw[dotted]   (0,1.6) --
                ++(-0.5,0) --
                ++(0,-0.5);
\draw (3,0.6) -- 
      ++(-0.5,0) --
      ++(0,0.5);
\draw[dotted]   (3,0.6) --
                ++(0.5,0) --
                ++(0,0.5);
\draw (3,1.6) -- 
      ++(0.5,0) --
      ++(0,-0.5);
\draw[dotted]   (3,1.6) --
                ++(-0.5,0) --
                ++(0,-0.5);

\filldraw  (-3.5,0.6) circle (1pt) 
                    node[anchor=east] {\small $v_{z_0}$}
           (-2.5,1.6) circle (1pt) 
                    node[anchor=west] {\small $\tilde v_{z_0}$}
           (-0.5,0.6) circle (1pt) 
                    node[anchor=east] {\small $v_{z_1}$}
           (0.5,1.6) circle (1pt) 
                    node[anchor=west] {\small $\tilde v_{\tilde z_0}$}
           (2.5,0.6) circle (1pt) 
                    node[anchor=east] {\small $v_{z_2} = \tilde v_{z_2}$};

\draw (-3.5,0.6) node[anchor=north] {\small $\star$}
      (0.5,1.6) node[anchor=south] {\small $\star$};

\draw[gray] (-3,0) node[anchor=south] {\footnotesize $\phantom{_{z_0}}\Delta_{z_0}$}
            (0,0) node[anchor=south]  {\footnotesize $\phantom{_{z_0}}\Delta_{z_1}$}
            (3,0) node[anchor=south]  {\footnotesize $\phantom{_{z_0}}\Delta_{z_2}$};

\draw[-latex,dashed] (-3.45,0.75) to[bend left] (-2.6,1.55);
\draw[-latex] (-2.55,1.45) to[bend left] (-3.4,0.65);

\draw[-latex] (-0.45,0.75) to[bend left] (0.4,1.55) ;
\draw[-latex,dashed] (0.45,1.45) to[bend left] (-0.4,0.65);


\end{tikzpicture}
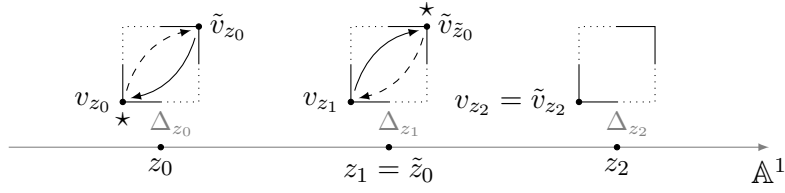
\\ \small{\textsc{Fig. 4}}
\end{center}
\vspace{0.5em}

As a corollary of Proposition~\ref{th_3_1} we obtain

\begin{corollary}
Adjacent homogeneous locally nilpotent derivations of horizontal type $\df_{e}$ and $\df_{\tilde e}$ on the algebra $A = A[Y,\mf{D}]$ commute if and only if they are coherent and $\df_{\tilde e}$ is simple in coordinates associated with $\df_{e}$ and $\df_{\tilde e}$.
\end{corollary}
\begin{proof}
Since $\df_{\tilde e}$ is simple then there are two possibilities. If $\tilde v_z = 0$ for all $z \ne 0$ then the first condition of the proposition holds. The other possibility is that there exists a point $z_1$ such that $\tilde v_{z_1} \ne 0$. Using $v_{z_1} = 0$, we obtain coherency equalities for $z_1$:  $\tilde v_{z_1}(e) = 1+0$, $0 = 1+ \tilde  v_{z_1}(\tilde e)$. These last equalities lead to the second condition of proposition~\ref{th_3_1}. 
\end{proof}

Further we need the following technical lemma. The proof of the lemma is beyond the scope of our discussion and is given in Appendix.

\begin{lemma}\label{lm4}
Let $z_i \in \mathbb{K}$, $z_i \ne z_j$ for $i\ne j$, $a_k \in \mathbb{Z}$. Then $$\sum_{k = 1}^l  \dfrac{a_k}{t-z_k} = \lambda = \text{const}$$ if and only if $\lambda = 0$ and $a_k = 0$ for all $k$.
\end{lemma}

Now we are ready to prove the main theorem

\begin{theorem}\label{main_result}
Homogeneous locally nilpotent derivations of horizontal type $\df_{e}$ and $\df_{\tilde e}$ on the algebra $A = A[Y,\mf{D}]$ commute if an only if they are coherent and $\df_{\tilde e}$ is simple in coordinates associated with $\df_{e}$ and $\df_{\tilde e}$.
\end{theorem}

\begin{proof}
We consider two cases. The first one is the most difficult. We prove the necessity in both cases. The sufficiency follows from substituting the equations obtained at the end of each case.

\emph{Case 1: $z_0 \ne 0$.}

We suppose that $v_{z_0} \notin N$ and $\tilde v_0 \notin N$ (particularly, the vertex $v_{z_0} \ne 0$ and $\tilde v_0 \ne~0$). Otherwise one of the derivations has no marked points and any point can be chosen as being marked. Thus, we may assume $z_0 = 0$.

The commutator of derivations is zero if and only if its application to  $t$ and $\chi$ is zero. By Lemma~\ref{lm3} the application of the commutator to $t$ is equivalent to the fact that (\ref{dt}) vanishes as a polynomial of $t$.

Let us extract terms with $z = z_0$ from $- (\alpha_e+\alpha_{\tilde e})q$ in~(\ref{dt}): 
$$
\Bigl(\tilde{v}_{z_0}(e+\tilde e)\dfrac{t}{t-z_0}\Bigr)(t-z_0) = t \cdot \tilde{v}_{z_0}(e+\tilde e).
$$
By horizontal line above alpha we denote the sum without term with $z = z_0$:  
$$
\bar\alpha_e = -t\cdot\sum_{z_k \ne 0,\, z_0} \tilde{v}_{z_k}(e)\dfrac{\mu_k(t)}{\mu(t)}. 
$$
Then from (\ref{dt}) it follows that
$$
B_{\phantom{0}} = (\tilde{v}_0(e)-\tilde s - v_{z_0}(\tilde e)+s+\tilde{v}_{z_0}(e+\tilde e))t - (\tilde{v}_0(e)-\tilde s-1)z_0 - (\bar\alpha_e+\bar\alpha_{\tilde e})q.
$$
If $B$ is zero then $B(0)=0$ and $B(z_0)=0$. Then using $\bar\alpha_e(0)=0$ we obtain
\begin{align}
 \tilde{v}_0(e)-\tilde s - 1 &= 0 \label{lt1}\\
 1 - v_{z_0}(\tilde e)+s+\tilde{v}_{z_0}(e+\tilde e) &= 0.\label{lt2}
\end{align}
Now let us consider $[\df_{e},\df_{\tilde e}](\chi^m) = 0$. We repeat the previous procedure with (\ref{C_1}) and (\ref{C_2}). Further, $C_1, C_2$ will have the form
\begin{align*}
  C_{1} &= 
 -\tilde{v}_0(m)\tilde{v}_{z_0}(\tilde e)tq+\tilde s \tilde{v}_{z_0}(m) tq + v_{z_0}(\tilde e)\tilde{v}_{z_0}(m) t^2 - v_{z_0}(m)\tilde{v}_{z_0}(e)t^2+ \\
  &\phantom{==} \tilde{v}_0(m)\bar\alpha_{\tilde e}q^2-\tilde s\bar\alpha_m q^2 - v_{z_0}(\tilde e)\bar\alpha_m tq + v_{z_0}(m)\bar\alpha_e tq \\
 C_{2} &= -t z_0\tilde{v}_{z_0}(m)  - \tilde{v}_{z_0}(m)\tilde{v}_{z_0}(\tilde e)t^2
 -t\bar\alpha'_m q^2 - \bar\alpha_m\bar\alpha_{\tilde e}q^2.
\end{align*}
The condition $[\df_{e},\df_{\tilde e}](\chi^m) = 0$ is equivalent to $F(t) := C_0+C_1+C_2 = 0$. Hence, $F(0)=0$ and $F(z_0)=0$ implies the following equalities:
\begin{align}
 \tilde{s} \tilde{v}_0(m) &= 0 \label{lx1}\\
 -v_{z_0}(m)\bigl(\tilde v_{z_0}(e) + s\bigr) - \tilde{v}_{z_0}(m)\bigl(-v_{z_0}(\tilde e)+1+\tilde{v}_{z_0}(\tilde e)\bigr) &= 0.\label{lx2}
\end{align}
From (\ref{lx1}) it follows that $ \tilde{s} = 0$. By substituting $ \tilde{s} = 0$ in (\ref{s})  and (\ref{lt1}) we have 
\begin{align}\label{fin1}
\tilde{v}_0(\tilde e) = -1/\tilde d, \quad \tilde{v}_0(e) = 1.
\end{align}
Note that by (\ref{lt2}) expressions in brackets at $v_{z_0}(m)$ and $\tilde{v}_{z_0}(m)$ is equal in (\ref{lx2}). It follows that $\bigl(v_{z_0}(m)-\tilde{v}_{z_0}(m)\bigr)\bigl(\tilde v_{z_0}(e) + s\bigr) = 0$. Finally, by (\ref{lt2}) and (\ref{lx2}) we obtain
\begin{align}\label{fin2}
  \tilde v_{z_0}(e) + s = 0, \quad  v_{z_0}(\tilde e) - \tilde{v}_{z_0}(\tilde e) = 1.
\end{align}
Now we can simplify the equality $C_0 + C_1 + C_2 = 0$. Let us group terms without $\bar\alpha$ as in the following:
\begin{align*}
  -v_{z_0}(m)\bigl(st^2+tq+\tilde{v}_{z_0}(e)t^2\bigr)+
 \tilde{v}_{z_0}(m)\bigl( v_{z_0}(\tilde e) t^2  - \tilde{v}_{z_0}(\tilde e)t^2 -tz_0 \bigr)\\  +  \tilde{v}_0(m)\bigl( v_{z_0}(\tilde e)tq- \tilde{v}_{z_0}(\tilde e)tq  \bigr)
 =-v_{z_0}(m)tq + 
 \tilde{v}_{z_0}(m)\bigl( t^2  - tz_0 \bigr) +
 \tilde{v}_0(m)tq = \\ =
 \bigl(  \tilde{v}_{z_0}(m)-v_{z_0}(m) +\tilde{v}_0(m) \bigr)tq.
 \end{align*}
Thus, $C_0+C_1+C_2 = 0$ has the form
\begin{multline*}
\bigl(  \tilde{v}_{z_0}(m)-v_{z_0}(m) +\tilde{v}_0(m) \bigr)tq + \tilde{v}_0(m)\bar\alpha_{\tilde e}q^2-v_{z_0}(\tilde e)\bar\alpha_m tq + \\ + v_{z_0}(m)\bar\alpha_e tq -t\bar\alpha'_m q^2 - \bar\alpha_m\bar\alpha_{\tilde e}q^2 = 0.
\end{multline*}
Note that $\bar\alpha_m$ can be expressed as $\bar\alpha_m = t \cdot \bar{\bar\alpha}_m$, where $\bar{\bar\alpha}_m$ denotes the sum $-\sum_{z \ne 0,\, z_0} \tilde{v}_z(m) \dfrac{1}{t-z}$. Let us repeat this procedure with all terms containing $\bar\alpha$ in the last equation. Dividing by $tq$ we have
$$
\tilde{v}_{z_0}(m)-v_{z_0}(m) +\tilde{v}_0(m) + \tilde{v}_0(m)\bar{\bar\alpha}_{\tilde e}q-v_{z_0}(\tilde e)\bar{\bar\alpha}_m t  + v_{z_0}(m)\bar{\bar\alpha}_e t -\bar\alpha'_m q - \bar{\bar\alpha}_m\bar{\bar\alpha}_{\tilde e}q = 0.
$$
Substituting $0$ and $z_0$ for $t$, we obtain
\begin{align}
\tilde{v}_{z_0}(m)-v_{z_0}(m) +\tilde{v}_0(m) - z_0 \tilde{v}_0(m)\bar{\bar\alpha}_{\tilde e} + z_0\bar\alpha'_m = 0 \label{11}\\
\tilde{v}_{z_0}(m)-v_{z_0}(m) +\tilde{v}_0(m) - z_0 v_{z_0}(\tilde e)\bar{\bar\alpha}_m  + v_{z_0}(m)\bar{\bar\alpha}_e z_0 = 0.\label{22}
 \end{align}
Now if we recall Lemma~\ref{lm4} and group all terms with $\alpha$ in (\ref{22}) then we get $\tilde{v}_{z_0}(m)-v_{z_0}(m) +\tilde{v}_0(m) = 0$ for all $m \in M$, that is $v_{z_0} - \tilde{v}_{z_0}=\tilde{v}_0$. Note also that this implies $d=\tilde d$.

Then (\ref{11}) is given by $z_0 \cdot \tilde{v}_0(m)\bar{\bar\alpha}_{\tilde e} - z_0\cdot \bar\alpha'_m = 0$. Using (\ref{amm}) and dividing by $z_0$ we have
$$
\sum_{k = 1}^l  \tilde v_0(m)\tilde{v}_{z_k}(\tilde e)\dfrac{\mu_k(t)}{\mu(t)} -
\sum_{k = 1}^l  z_k\tilde{v}_{z_k}(m)\dfrac{\mu_k(t)^2}{\mu(t)^2} = 0.
$$
Then for the numerator of this fraction we obtain
$$
\sum_{k = 1}^l  \tilde v_0(m)\tilde{v}_{z_k}(\tilde e)\mu_k(t)\mu(t) -
\sum_{k = 1}^l  z_k\tilde{v}_{z_k}(m)\mu_k(t)^2 = 0.
$$
Substituting $z_j$ for $t$ we get
$$
z\tilde{v}_{z_k}(m)\mu_j(z_j)^2 = 0.
$$
Since $z_i \ne z_j$ for $i\ne j$, it follows that $\mu_j(z_j) \ne 0$. Consequently, $\tilde{v}_{z} = 0$ for $z \ne 0, \, z_0$. This prove that $\df_{\tilde e}$ is simple.

Finally, combining (\ref{fin1}) and (\ref{fin2}), we obtain necessity conditions:
\begin{align*}
\tilde{v}_0(\tilde e) = -1/d; \quad \tilde{v}_0(e) = 1; \quad
\tilde v_{z_0}(e) -v_{z_0}(e)  = 1/d; \quad v_{z_0}(\tilde e) - \tilde{v}_{z_0}(\tilde e) = 1.
\end{align*}
It is easy to see that these conditions are coherency equalities for derivations at points $z_0$ and $\tilde z_0 = 0$. Substitution of these conditions in (\ref{com-t}) and (\ref{com-chi}) proves the sufficiency. 

\emph{Case 2: $z_0 = 0$.}

For $ z_0 = 0$ expressions (\ref{com-t}) and (\ref{com-chi}) have the form
\begin{align*}
[\df_{e},\df_{\tilde e}](t) &= d \tilde d \cdot \ph^{\tilde e}\chi^{e+\tilde e} \cdot (\tilde{v}_0(e)-\tilde s - v_{0}(\tilde e)+s-\alpha_e-\alpha_{\tilde e})t \\
[\df_{e},\df_{\tilde e}](\chi^m) &= d\tilde d\cdot t^{s+\tilde s}  \cdot \ph^{\tilde e}\cdot \chi^{m+e+\tilde e} \bigl(C_0 + C_{1} + C_{2}\bigr),
\end{align*}
where
\begin{align*}
 C_{0} &= \tilde s\tilde{v}_0(m)  - sv_{0}(m) +\tilde{v}_0(m) v_{0}(\tilde e)- v_{0}(m)\tilde{v}_0(e)\\
 C_{1} &= \tilde{v}_0(m)\alpha_{\tilde e}-\tilde s\alpha_m  - v_{0}(\tilde e)\alpha_m  + v_{0}(m)\alpha_e  \\
 C_{2} &= -t\alpha'_m - \alpha_m\alpha_{\tilde e}.
\end{align*}
Consequently, the commutator is zero if and only if the following conditions hold:
\begin{equation}\label{sys0}
    \begin{cases}
    \tilde{v}_0(e)-\tilde s - v_{0}(\tilde e)+s-\alpha_e-\alpha_{\tilde e} = 0 \\
    C_0 + C_{1} + C_{2} = 0.
    \end{cases}
\end{equation}
Substituting $0$ for $t$ in both equations we obtain the necessity conditions:
\begin{equation}\label{sys1}
    \begin{cases}
    \tilde{v}_0(e)-\tilde s - v_{0}(\tilde e)+s = 0 \\
    \tilde s\tilde{v}_0(m)  - sv_{0}(m) +\tilde{v}_0(m) v_{0}(\tilde e)- v_{0}(m)\tilde{v}_0(e) = 0.
    \end{cases}
\end{equation}
Then, the first equalities of the systems (\ref{sys0}) and (\ref{sys1}) implies that $\alpha_e+\alpha_{\tilde e}=0$. By Lemma~\ref{lm4} it is equivalent to 
\begin{equation}\label{fin3}
 \tilde{v}_z(e+\tilde e) = 0, \quad    z \ne 0.
\end{equation}
The second equation of the system (\ref{sys1}) has the form 
$$
\tilde s\tilde{v}_0(m)  - sv_{0}(m) +\tilde{v}_0(m) v_{0}(\tilde e)- v_{0}(m)\tilde{v}_0(e) = \tilde v_0(m) \bigl(v_{0}(\tilde e)+\tilde s\bigr) - v_0(m)\bigl(\tilde{v}_0(e)+ s  \bigr).
$$
By the first equation of (\ref{sys1}) expressions in the brackets at $\tilde v_0(m)$ and $v_0(m)$ are equal. For (\ref{sys1}) let us substitute expressions in (\ref{s}) for $s$ and $\tilde s$ . Then we have
\begin{equation}\label{sys2}
    \begin{cases}
        \langle v_0-\tilde{v}_0, e+\tilde e \rangle = 1/ \tilde{d} - 1/d\\
        \langle\tilde v_0-v_0, m\rangle \cdot \bigl(\tilde{v}_0(e)+ s\bigr)=0.
    \end{cases}
\end{equation}

Further the second equation in (\ref{sys2}) implies two cases: $\tilde v_0 = v_0$ or $\tilde{v}_0(e)+ s = 0$. The first case corresponds to item 1 of Proposition~\ref{th_3_1}. Hence, it is enough to consider the second case to complete the proof.

Using (\ref{sys1}) we have $v_{0}(\tilde e)+\tilde s = 0$. Substituting expressions in (\ref{s}) for $s$ and $\tilde s$ we obtain 
\begin{equation}\label{fin4}
d\tilde v_0(e) = 1 + d v_0(e), \quad \tilde d v_0(\tilde e) = 1 + \tilde d v_0(\tilde e).
\end{equation}
Moreover, the second equation of (\ref{sys0}) has the form
\begin{equation*} \label{eq2}
\tilde{v}_0(m)\alpha_{\tilde e}-v_0(m)\alpha_e - t\alpha'_m - \alpha_m\alpha_{\tilde e} = 0.
\end{equation*}
As in the previous case of the theorem considering the numerator and substituting $z_j$ for $t$ we obtain $\tilde{v}_{z_j}(m)(1+\tilde{v}_{z_j}(\tilde e))=0$. Consequently, by~(\ref{ineq-v_z}) it follows that $\tilde{v}_{z_j}(\tilde e) = -1$. The coefficient at leading monomial should be zero. This implies
\begin{align*}
\sum_{k = 1}^l  \bigl(\tilde{v}_0(m)\tilde{v}_{z_k}(\tilde e)+v_0(m)\tilde{v}_{z_k}(e)\bigr) = \sum_{k = 1}^l\tilde{v}_{z_k}(m)\sum_{k = 1}^l\tilde{v}_{z_k}(\tilde e).
\end{align*}
From \label{fin3} it follows that $\tilde{v}_{z_k}(e) = -\tilde{v}_{z_k}(\tilde e) = 1$. Hence
\begin{align}\label{eq5}
l \cdot \Bigl(\tilde{v}_0(m)-v_0(m) - \sum_{k = 1}^l\tilde{v}_{z_k}(m)\Bigr)=0.
\end{align}
Thus, if $l = 0$ we have case 1 of Theorem~\ref{th_3_1}. Otherwise the expression in brackets in (\ref{eq5}) is zero. Then  $\tilde{v}_0-v_0 = \sum_{k = 1}^l\tilde{v}_{z_k} \in N$. Hence $d=\tilde d$ and
$$
-1/d = \langle \tilde{v}_0- v_0, \tilde e\rangle = \langle \sum_{k = 1}^l\tilde{v}_{z_k}, \tilde e\rangle = \sum_{k = 1}^l\tilde{v}_{z_k}( \tilde e) \in \ZZ.
$$
Consequently, $d=1$ and there exists a unique point $z_1$, such that $\tilde{v}_{z_1} \ne 0$ and $\tilde{v}_{z_1}(\tilde e) = -1$. It follows that $\df_{\tilde e}$ is simple. Using the previous equations we have $\tilde{v}_{z_1}(e) = 1$. Therefore, we obtain coherency equalities at the point $z_1$. Finally, coherency equalities at the point $\tilde z_0 = 0$ is given by (\ref{fin4}). This completes the proof.
\end{proof}

Combining all cases being described in the theorem we obtain

\begin{corollary}\label{crl_main}
Homogeneous locally nilpotent derivations of horizontal type $\df_{e}$ and $\df_{\tilde e}$ commute if and only if one of the following five conditions hold in associated generators system.
\begin{enumerate}
    \item $\tilde v_z = 0$ for $z \ne 0$, $z_0 = 0$,  $\tilde z_0 = 0$ and
    \begin{enumerate}
        \item $v_0=\tilde v_0;$ or
        \item $d\tilde v_0(e) = 1 + d v_0(e)$, $ \tilde d v_0(\tilde e) = 1 + \tilde d \tilde v_0(\tilde e);$
    \end{enumerate}
    \item $\tilde v_z = 0$ for $z \ne 0$ and $z\ne z_1$, $z_0 = 0$, $\tilde z_0 = 0$ \\
     $v_0 \in N$, $\tilde v_0 \in N$, $\tilde v_{z_1}(e) = 1$, $\tilde v_{z_1}(\tilde e) = -1;$ and
    \begin{enumerate}
        \item $v_0=\tilde v_0;$ or
        \item $\tilde v_0(e) = 1 + v_0(e)$, $v_0(\tilde e) = 1 + \tilde v_0(\tilde e);$
    \end{enumerate} 
    \item $\tilde v_z = 0$ for $z \ne 0$ and $z\ne z_0$, $z_0 \ne 0$, $\tilde z_0 = 0$ and\\
    $d \tilde v_{z_0}(e) = 1 + d v_{z_0}(e)$, $v_{z_0}(\tilde e) = 1 + \tilde v_{z_0}(\tilde e)$, $\tilde v_0(e) = 1$, $d \tilde v_0(\tilde e) = -1$, $d = \tilde d.$
\end{enumerate}
\end{corollary}

\begin{example}
Let us illustrate the result on Example~\ref{ex01d}. Recall that we consider a variety defined by proper polyhedral divisor  $\mf{D} = [0,1]\cdot\{0\}$ over affine line. Besides, homogeneous locally nilpotent derivations of horizontal type corresponding to points $0$ and $1$ give two families of derivations:
\begin{align*}
    \df_0(t^r\chi^m) &=  r t^{r-1} \chi^{m+e} \\
    \df_1(t^r\chi^m) &= (m+r) t^{r+\tilde s} \chi^{m+\tilde e}.
\end{align*}
Since $z_0=\tilde z_0 = 0$, the example corresponds to item 2 of Corollary~\ref{crl_main}. The derivations under consideration correspond to different marked vertices. Hence, they commute if and only if coherency equalities hold. Then we have $e = 1 + 0$ and $0 = 1 + \tilde e$, that is $e=1$, $\tilde e = -1$ and $\tilde s = 0$. Thus, derivations in variables $x, y$ commute if and only if $a=b=0$ in (\ref{ex1}). Besides, if both derivations belong to one family then they commute because they have the same marked vertices. Corresponding $\Ga^2$-actions have the form
\begin{align*}
   &(1) \quad x \mapsto x + \l, \quad y \mapsto y + \m \\
   &(2) \quad x \mapsto x, \quad y \mapsto y + \l x^{a_1} + \m x^{a_2}\\
   &(3) \quad x \mapsto x + \l y^{b_1} + \m y^{b_2}, \quad y \mapsto y.
\end{align*}
\end{example}

\begin{example}
Now let us consider the derivations from Example~\ref{ex02d}. Note again that if both derivations belong to one family then they commute because they have the same marked vertices. If derivations belong to families (1) and (3) respectively then the corresponding vertices of polyhedron $\Delta_1$ are equal. Coherency equations for $\Delta_0$ have the form $a_1+4b_1 = -4, a_3+4b_3 = 1$. Then we obtain
\begin{align*}
    &(1) \:\: \df_1 = x_3^{b_1} \frac{\df}{\df x_1} - x_3^{b_1-1} (2 x_1 x_2^4 + 1)\frac{\df}{\df x_4} , b_1 \ge 1 \\ 
    &(3) \:\: \df_3 = x_3^{b_3}\frac{\df}{\df x_2} - 4 x_1^2 x_2^3 x_3^{b_3-1} \frac{\df}{\df x_4}, b_3 \ge 1. 
\end{align*}
The corresponding $\Ga^2$-action has the form
\begin{align*}
    & x_1 \mapsto x_1 + \l \cdot x_3^{b_1}, \quad x_2 \mapsto x_2 + \m \cdot x_3^{b_3}, \quad x_3 \mapsto x_3 \\
    & x_4 \mapsto x_4  -  \m \cdot 4 x_1^2 x_2^3 x_3^{b_3-1} - \l \cdot x_3^{b_1-1} \bigl(2 x_1 x_2^4 + 1 + 8 \m \cdot  x_1 x_2^3 x_3^{b_3} \bigr) + \\
    &- \l \cdot x_3^{b_1-1} \bigl(12 \m^2 \cdot x_1 x_2^2 x_3^{2b_3} + 4 \m^3 \cdot x_1 x_2 x_3^{3b_3} + \frac{\m^4}{6}\cdot x_1 x_3^{4 b_3} \bigr) ,\\
    \:& b_1 \ge 1, \: b_3 \ge 1.
\end{align*}
If derivations belong to families (2) and (4) respectively then coherency equations have the form $a_2+8 b_2 = -4, a_4 + 8 b_4 = 1$. Corresponding $\Ga^2$-action are the following:
\begin{align*}
    & x_1 \mapsto x_1 - \l \cdot x_4^{-b_2},
    \quad x_2 \mapsto x_2 + \m \cdot x_4^{-b_4}, \quad x_4 \mapsto x_4,
    \\& x_3 \mapsto x_3 + 4\m \cdot x_1 x_2 x_3^{-b_4-1} + \l \cdot  x_4^{-b_2-1} \bigl(2 x_1 x_2^4 + 1 + \\ 
    & + 8 \m \cdot x_1 x_2^3 x_4^{-b_4} + 12 \m^2 \cdot x_1 x_2^2 x_4^{-2b_4} + 4 \m^3 \cdot x_1 x_2 x_4^{-3b_4} + \frac{\m^4}{6}\cdot x_1 x_4^{-4b_4} \bigr) , \\
    & 
    \quad b_2 \le -1, \: b_4 \le -1.
\end{align*}
Applying Theorem~\ref{main_result} to pairs (1) and (2), (1) and (4), (2) and (3), (3) and (4), we obtain that in cases (1) and (4), (2) and (3) derivations do not commute. For the other cases similar calculations can be performed.
Therefore, we obtain the actions of the group $\TT^2 \rightthreetimes \mathbb{G}_a^2$ on the hypersurface (\ref{ex2-alg-iso}) with open orbit.%
\end{example}%

\section{Appendix}\label{s8}

\begin{prooff}[\textsc{of Lemma \ref{lm3}}]
We compute the derivation on functions $t$, $q$, $\chi^m$, $\ph^m$. Substituting $0$ for $m$ and $r$ in (\ref{f-hor}) we have
\begin{align*}
\df_e(t) &= \df_e(q) = dq^{s+1}\chi^e \\
\df_{\tilde e}(t) &= \df_{\tilde e}(q) = \tilde d t^{\tilde s+1}\ph^{\tilde e}\chi^{\tilde e} \\
\df_e(\chi^m) &= dv_{z_0}(m)q^s\chi^{m+e}.
\end{align*}
Using the chain rule for derivations and previous notation $\alpha_m = t\dfrac{(\ph^m)'}{\ph^m}$ we have
\begin{align*}
\df_e(\ph^m) &= (\ph^m)'\df_e(t) =d(\ph^m)'q^{s+1}\chi^e\\
\df_{\tilde e}(\ph^m) &= (\ph^m)'\df_{\tilde e}(t) = (\ph^m)'\cdot \tilde d t^{\tilde s+1}\ph^{\tilde e}\chi^{\tilde e} =\tilde d\alpha_m t^{\tilde s} \ph^{m+\tilde e}\chi^{\tilde e}.
\end{align*}
\noindent By Leibniz rule we obtain $\df_{\tilde e}(\ph^m\chi^m)=\df_{\tilde e}(\ph^m)\chi^m + \ph^m\df_{\tilde e}(\chi^m) = \tilde{d}\tilde{v}_{0}(m) \cdot t^{\tilde{s}} \cdot \ph^{m+\tilde{e}}\chi^{m+\tilde{e}} $. It follows that
$$
\df_{\tilde e}(\chi^m) = \tilde d (\tilde{v}_0(m)-\alpha_m)t^{\tilde s}\ph^{\tilde e}\chi^{m+\tilde e}.
$$
Then the commutator of derivations on function $t$ has the form
\begin{multline*}
[\df_{e},\df_{\tilde e}](t) = \df_{e}(\df_{\tilde e}(t))-\df_{\tilde e}(\df_{e}(t)) = \df_{e}(\tilde d t^{\tilde s+1}\ph^{\tilde e}\chi^{\tilde e}) - \df_{\tilde e}(dq^{s+1}\chi^e) = \\ \tilde d (\tilde s +1) t^{\tilde s}\cdot \ph^{\tilde e} \chi^{\tilde e}\cdot \df_e(t) + \tilde d t^{\tilde s +1} \cdot \chi^{\tilde e}\cdot \df_e(\ph^{\tilde e}) + \tilde d t^{\tilde s +1} \cdot \ph^{\tilde e}\cdot \df_e(\chi^{\tilde e}) + \\ -d(s+1)\cdot q^s\cdot \chi^e \cdot \df_{\tilde e}(q) - d\cdot q^{s+1} \cdot \df_{\tilde e}(\chi^e).
\end{multline*}
Using the previous equalities we have
\begin{multline*}
[\df_{e},\df_{\tilde e}](t) = d\tilde d  \cdot (\tilde s +1) \cdot t^{\tilde s} q^{s+1}\cdot \ph^{\tilde e} \chi^{e+\tilde e} + d\tilde d \cdot t^{\tilde s}q^{s+1} \cdot \ph^{\tilde e}\chi^{\tilde e}\cdot \alpha_{\tilde e} + \\ + d\tilde d\cdot v_{z_0}(\tilde e) \cdot t^{\tilde s +1}q^s \cdot \ph^{\tilde e}\chi^{e+\tilde e} -d\tilde d\cdot(s+1)\cdot t^{\tilde s+1}\cdot q^s\cdot \ph^{\tilde e} \chi^{e+\tilde e} \cdot \df_{\tilde e}(q) + \\ - d\tilde d\cdot(\tilde v_0(e)-\alpha_e)\cdot t^{\tilde s} q^{s+1} \cdot \ph^{\tilde e} \chi^{e+\tilde e}.
\end{multline*}
The last step is factoring out the term $d \tilde d \cdot t^{\tilde s} q^s \cdot \ph^{\tilde e}\chi^{e+\tilde e}$.
Similarly using Leibniz rule we compute $[\df_e,\df_{\tilde e}](\chi^m)$ and obtain
\begin{align*}
[\df_e,\df_{\tilde e}](\chi^m) &= \df_{e}(\df_{\tilde e}(\chi^m))-\df_{\tilde e}(\df_{e}(\chi^m)) \\
=  \df_{e}(\tilde d (\tilde{v}_0(&m)-\alpha_m)t^{\tilde s}\ph^{\tilde e}\chi^{m+\tilde e}) - \df_{\tilde e}(dv_{z_0}(m)q^s\chi^{m+e})= \\
&=\tilde d \tilde{v}_0(m)\Bigl(\tilde s t^{\tilde s-1}dq^{s+1}\chi^e\Bigr)\ph^{\tilde e}\chi^{m+\tilde e}+\tilde d \tilde{v}_0(m)t^{\tilde s}\Bigl((\ph^{\tilde e})'dq^{s+1}\chi^e\Bigr)\chi^{m+\tilde e}+ \\
&+\tilde d \tilde{v}_0(m)t^{\tilde s}\ph^{\tilde e}\Bigl(dv_{z_0}(m+\tilde e)q^s\chi^{m+e+\tilde e}\Bigr)-\tilde d \alpha_m\Bigl(\tilde s t^{\tilde s-1}dq^{s+1}\chi^e\Bigr)\ph^{\tilde e}\chi^{m+\tilde e}+\\
&-\tilde d \alpha_m t^{\tilde s}\Bigl((\ph^{\tilde e})'dq^{s+1}\chi^e\Bigr)\chi^{m+\tilde e}-\tilde d \alpha_m t^{\tilde s}\ph^{\tilde e}\Bigl(dv_{z_0}(m+\tilde e)q^s\chi^{m+e+\tilde e}\Bigr)+\\
&-\tilde d \Bigl(\alpha'_m dq^{s+1}\chi^e\Bigr)t^{\tilde s}\ph^{\tilde e}\chi^{m+\tilde e}- d v_{z_0}(m)\Bigl(sq^{s-1} \tilde d t^{\tilde s+1}\ph^{\tilde e}\chi^{\tilde e}\Bigr)\chi^{m+e}+\\
&- d v_{z_0}(m) q^{s}\Bigl(\tilde d (\tilde{v}_0(m+e)-\alpha_{m+e})t^{\tilde s}\ph^{\tilde e}\chi^{m+e+\tilde e}\Bigr).
\end{align*}
After factoring out the term $d\tilde d\cdot t^{\tilde s-1} q^{s-1} \cdot \ph^{\tilde e}\cdot \chi^{m+e+\tilde e}$ we have what is needed. This completes the proof.
\end{prooff}

\begin{prooff}[\textsc{of Lemma \ref{lm4}}]
By $\mu(t)$ denote $\prod_{k=1}^l (t-z_k)$, $\mu_k(t) = \prod_{i\ne k} (t-z_i)$. Then we obtain
$$
\sum_{k = 1}^l  a_k\dfrac{\mu_k(t)}{\mu(t)} - \lambda = 0.
$$
The fraction vanishes if and only if the numerator vanishes:
$$
\sum_{k = 1}^l  a_k\mu_k(t) - \lambda \cdot \mu(t) = 0 .
$$
Substituting $z_j$ for $t$ and using the fact that $\mu_k(z_j) = 0$ for $j \ne k$ we have
$$
a_j\mu_j(z_j) = a_j  \prod_{j \ne i}(z_j-z_i) = 0.
$$
$z_i \ne z_j$ for $i \ne j$. It follows that $a_j=0$ for all $j$ and then $\lambda = 0$. This concludes the proof.
\end{prooff}

\end{document}